\documentclass {article}
\usepackage[english]{babel}
\usepackage[utf8]{inputenc}
\usepackage[T1]{fontenc}
\usepackage{graphicx}
\usepackage[final]{pdfpages}
\usepackage[letterpaper, margin=1in]{geometry}
\usepackage[extra]{tipa}
\usepackage{titlesec}
\usepackage{xspace}		
\usepackage{mathtools}
\usepackage[ruled,vlined]{algorithm2e}
\usepackage{amsmath,amsfonts}
\usepackage{dsfont}
\usepackage{amssymb}
\usepackage[ntheorem]{empheq}
\usepackage{mathenv}
\usepackage[standard,thref,amsmath,thmmarks]{ntheorem}

\usepackage{thmtools,hyperref}
\usepackage{nicefrac}

\usepackage{theoremref}

\usepackage{caption,subcaption}
\usepackage{color}
\usepackage{comment}
\usepackage{amsfonts}
\usepackage{bm}

\usepackage{minitoc}
\usepackage{cite}
\mathtoolsset{showonlyrefs,showmanualtags}

\numberwithin{counter}{subsection}

\usepackage{aliascnt}

\renewtheorem{theorem}{Theorem}[section]

\newaliascnt{lem}{theorem}
\renewtheorem{lemma}[lem]{Lemma}
\aliascntresetthe{lem}

\newaliascnt{def}{theorem}
\renewtheorem{definition}[def]{Definition}
\aliascntresetthe{def}

\newaliascnt{prop}{theorem}
\renewtheorem{proposition}[prop]{Proposition}
\aliascntresetthe{prop}

\newaliascnt{coro}{theorem}
\renewtheorem{corollary}[coro]{Corollary}
\aliascntresetthe{coro}

\newaliascnt{rmk}{theorem}
\renewtheorem{remark}[rmk]{Remark}
\aliascntresetthe{rmk}

\newcommand{\de}{\mathrm{d}}
\newcommand{\D}{\partial}
\newcommand{\eps}{\varepsilon}
\newcommand{\R}{\mathbb{R}}

\title{Dynamics of concentration in a population structured by age and a phenotypic trait with mutations. Convergence of the corrector}
\author{
Samuel Nordmann\thanks{\'Ecole des Hautes \'Etudes en Sciences Sociales, CAMS,
54 boulevard Raspail, 75006 Paris, France. Email: samuel.nordmann@ehess.fr}
\and
Beno\^ \i t Perthame\thanks{Sorbonne Universit\'e, CNRS, Universit\'e de Paris, Inria, Laboratoire Jacques-Louis Lions, 75005 Paris, France. Email: benoit.perthame@upmc.fr}
\thanks{\textbf{Acknowledgment.} BP has received funding from the European Research Council (ERC) under the European Union’s Horizon 2020 research and innovation programme (grant agreement No 740623)}
}

\begin{document}

\maketitle

\begin{abstract}
We study an equation structured by age and a phenotypic trait describing the growth process of a population subject to aging, competition between individuals, and mutations.  This leads to a renewal equation which occurs in many evolutionary biology problems. We aim to describe precisely the asymptotic behavior of the solution, to infer properties that illustrate the concentration and adaptive dynamics of such a population.
This work is a continuation of \cite{Nordmann2018a} where the case without mutations is considered. When mutations are taken into account, it is necessary to control the corrector which is the main novelty of the present paper.

Our approach consists in defining, by the Hopf transform, a Hamilton-Jacobi equation with an effective Hamiltonian as in homogenization problems. Its solution carries the singular part of the limiting density (typically Dirac masses) and the corrector defines the weights. The main new result of this paper is to prove that the corrector is uniformly bounded, using only the global Lipschitz and semi-convexity estimates for the viscosity solution of the Hamilton-Jacobi equation. We also establish the limiting equation satisfied by the corrector. To the best of our knowledge, this is the first example where such bounds can be proved in such a context. 
\end{abstract}

 \\[4mm]
\noindent {\bf AMS Class. No:} 	35B40, 35F21, 35Q92, 49L25. 

\section{Introduction}
\subsection{Main results}
We study a mathematical model describing the growth process of a population structured by age and a phenotypic trait, subject to aging, competition between individuals and mutations. Our goal is to describe the asymptotic behavior of the solution, in particular the selection of the fittest traits and the adaptative dynamics of such traits. For  $\eps>0$, we choose $m_\eps(t,x,y)$ to represent the population density of individuals who, at time $t\geq0$, have age $x\geq0$ and a quantitative phenotipic trait $y\in\R^n$, solution of  a renewal type equation
\begin{equation}\label{equa}
\left\{\begin{array}{ll}
\eps \D_t m_\eps + \D_x\left[A(x,y)m_\eps\right] +\left(\rho_\eps(t)+d(x,y)\right)m_\eps=0,
\\[2mm]
A(0,y)m_\eps\left(t,0,y\right)=\frac{1}{\eps^n}\int_{\R^n}\int_{\R_+}{M(\frac{y'-y}{\eps})b(x',y')m_\eps(t,x',y')dx'dy'}, 
\\[2mm]
\rho_\eps(t)=\int_{\R_+}\int_{\R^n}{m_\eps(t,x,y)dx dy},
\\[2mm]
m_\eps(t=0,x,y)=m_\eps^0(x,y)>0.
\end{array}\right. 
\end{equation}
The parameter $\eps>0$ stands for a hyperbolic rescaling $(t,y)\leftrightarrow(\eps^{-1}t,\eps^{-1}y)$ in the mutation term. Our main concern is to study the asymptotics of $m_\eps$ when $\eps\to0$.

This work is the continuation of the study begun in \cite{Nordmann2018a}, where the model without mutations is studied and where we proved that there is a measure $\mu$ (typically a Dirac mass) and a bounded profile ${\mathcal Q}$ such that
$$
m_\eps(t,x,y) \rightharpoonup {\mathcal Q}(t,x,y)\, \mu(t,y),
$$
in other words, the asymptotic singularity is carried in the variable $y$ only.
In the present work, the mutation term in the second line of~\eqref{equa} adds a significant difficulty because the profile  ${\mathcal Q}$ turns out to be strongly related to the limit of the corrector in the spectral problem defining the underlying effective Hamiltonian, i.e., the \emph{effective fitness} in our context, arising in a  Hamilton-Jacobi equation which defines the singular part.

While the classical approach consists in studying the asymptotics of $v_\eps(t,x,y):=\eps \ln (m_\eps(t,x,y))$, here, and following \cite{BP:PS-dispersal, Nordmann2018a}, we define a priori some kind of variable separation setting 
$$
m_\eps(t,x,y):= p_\eps(t,x,y) e^\frac{U_\eps(t,y)- \int_0^t \rho_\eps(\cdot)}{\eps}
$$
where the exponential term carries the singular part of the limiting population density and $U_\eps$ is defined autonomously through a Hamilton-Jacobi equation (see~\eqref{DefU_eps}) involving the \emph{effective fitness} $\Lambda$ (defined in~\eqref{EigenProblemFinal}). The main finding of this paper is the proof that, with our choice of $U_\eps$, the corrector $p_\eps$ 
 is uniformly bounded. We also show that $p_\eps$ converges to a multiple of the principal eigenfunction  $Q(t,x,y)$ of the formal limiting operator (defined in~\eqref{EigenProblemFinal}).
It justifies that the ansatz $U_\eps$ is appropriate and carries all the information on the singular behavior of $m_\eps$ when $\eps\to0$.
The formal idea of the method is explained with more details in section~\ref{SecFormalApproach} and consists in keeping as simple as possible the Hamilton-Jacobi equation which defines $U_\eps$  while all the functional analytic difficulties are carried by the corrector $p_\eps$ which satisfies a linear equation. 

The first step is to define  $U_\eps$ and study its properties.
\begin{theorem}[Convergence of $U_\eps$] \label{th:MainResultsConvergenceU}
Under the assumptions of section~\ref{SecAssumptions}, when $\eps\to0$, the ansatz $U_\eps(t,y)$ is well defined through~\eqref{DefU_eps} and converges in $W_{loc}^{1,r}$, $1 \leq r < \infty$, to some $U\in W^{1,\infty}_{loc}$ which is semi-convex and is a viscosity solution of the Hamilton-Jacobi equation~\eqref{DefU}.
\end{theorem}
The second step is to prove that the corrector is uniformly bounded and converges.
\begin{theorem}[Convergence of the corrector] \label{th:MainResultsConvergenceP}
Under the assumptions of section~\ref{SecAssumptions}, for any fixed $T>0$, $p_\eps(t,x,y)$ and $\int_{x>0}p_\eps(t,x,y)dx$ are  bounded from above and below, uniformly in $\eps>0$, $t\in[0,T]$, $y\in\R^n$.

In addition, for $t\in[0,T]$, $x\in[0,\overline{x}]$, $y\in\R^n$, up to extraction of a subsequence $\eps\to0$, $p_\eps$ converges in $L^\infty$-weak-$^\star$ to $\gamma(t,y) Q(t,x,y)$, where $Q$  is the eigenfunction defined in~\eqref{EigenProblemFinal}  and $\gamma\in L^\infty$ formally satisfies~\eqref{EquationVFinale}.
\end{theorem} 

With these informations, it is standard that the \emph{concentration points} of the population density $m_\eps(t,x,y)$ are carried by the set
\begin{equation}\label{DefSetS}
\mathcal{S}:=\left\{t\geq0,\ y\in\R^n:U(t,y)=\sup\limits_{y'\in\R^n}U(t,y')\right\}.
\end{equation}
\begin{theorem}[concentration]\label{MainResults}
Under the assumptions of section~\ref{SecAssumptions}, when $\eps\to0$
\\
1. The total population $\rho_\eps$ converges weakly to some positive $\rho\in L^\infty$ and
\begin{equation}
\forall t>0,\quad \int_0^t \rho = \sup\limits_{y\in\R^n} U(t,y).
\end{equation}
2. The population $m_\eps$ vanishes locally uniformly outside the set $\mathcal{S}$.
\\
3. Under further assumptions on the initial conditions, and for small times $t\in[0,T]$, we have
$$
\mathcal{S}=\{(t,\bar y(t))\},
$$
where $\bar y(t)$ follows the Canonical Equation 
\begin{equation}
\left\{\begin{aligned}
&\frac{\mathrm{d}}{\mathrm{d}t}\bar y(t)=\left(D_y^2U(t,\bar y(t))\right)^{-1}\cdot\nabla_y\Lambda(\bar y(t),1)+\D_\eta\Lambda(\bar y(t),1)\int_{\R^n}M(z)z dz,\\
&\bar y(0)=\bar y^0.
\end{aligned}\right.
\end{equation}
\end{theorem}
Notice that we establish convergence of the full families $U_\eps$ and $\rho_\eps$, without using the famous uniqueness result of~\cite{CaLa}, because of the simple dependency on the unknown $\rho_\eps$ in our setting.

We point out that the main restriction on the coefficients is the assumption that the transport in $x$ outwards the support of $b(\cdot,y)$ occurs in finite time, see~\eqref{AssumptionIntegralACombined}. The other assumptions, detailed in section~\ref{SecAssumptions}, are formulated directly on the limiting eigenproblem and are quite general.


\subsection{The model}

Some possible biological interpretations of the model~\eqref{equa} are as follows. The function $A(x,y)$ is the speed of aging of individuals with age $x$ and phenotypic trait $y$.
The total size of the population at time $t$ is denoted with $\rho_\eps(t)$. In the mortality term, $d(x,y)>0$ represents intrinsic death rate  and the nonlocal term $\rho_\eps(t)$ represents competition.
The condition at the boundary $x=0$ describes the birth of newborns that happens with rate $b(x,y)>0$  and with the probability kernel~$M$ for mutations.  

The terminology of \emph{renewal equation} comes from this boundary condition. It is related to the McKendrick-von Foerster equation which is only structured in age (see \cite{perthame} for a study of the linear equation). This model has been extended with other structuring variables, as size~\cite{DiekPhys, Mis-Pe-Ry}, DNA content, maturation, etc., in the context of cell divisions~\cite{Doumic-Gabriel}, or proliferative and quiescent states of tumor cells~\cite{Adimy-Cr-Ru, GyllWebb}. Space structured problems have also been extensively studied~\cite{Jabin2016, Mi-patch, Mi-Pe-spatial,BP:PS-dispersal}. 

To keep the model~\eqref{equa} quite general, we have allowed the progression speed $A$ to depend on $x$. Thus, although the variable~$x$ is referred to as \emph{age}, it can represent other biological quantities that evolve throughout the individual lifespan such as, for instance, the size of individuals, a physiological age, a parasite load, etc. 

The rescaling parameter $\eps > 0$ comes from a hyperbolic rescaling of $t$ and $y$. Accordingly, the dynamics are considered in two different time scales. The first one is the individual lifetime scale $\eps t$, i.e., the characteristic time for the population to reach the dynamical equilibrium for a fixed $y$. The second one is the evolutionary time scale $t$, corresponding to the evolution of the population distribution with respect to the variable $y$. Formally, at the limit when $\eps\to0$, the time scales are completely separated. 
This rescaling is a classical way to give a continuous formulation of the adaptive evolution of a phenotypically structured population (see \cite{CFBA, OD, Di-Ja-Mi-Pe, AL.SM.BP}). Note that the mutation kernel is supposed to be \emph{thin-tailed}, i.e., it decreases faster than any exponential. A fat-tailed kernel needs a different rescaling, see \cite{Bouin}.

From the modelization point of view, \autoref{MainResults} is a form of mathematical formulation of Natural Selection and Evolution. On an ecological time scale, only the phenotype $\bar y(t)$ which maximizes the ecological fitness $U(t,\cdot)$ can survive. On an evolutionary time scale, we observe the dynamics of $\bar y(t)$.
Similar results as \autoref{MainResults} have been obtained for various models with parabolic equations~\cite{GB.BP:07, GB.SM.BP:09, AL.SM.BP} and integrodifferential equations~\cite{Ba-Pe, DJMR, lorenzi-2013}. More generally, convergence to positive measures in selection-mutation models has been studied by many authors~\cite{ackleh_F_T, calsina-al-2013, Busse2016}. The special case of age-structured populations are also considered in~\cite{Meleard2009,Tran2008,VC:PG:AMG}.

\subsection{A formal presentation of the method}
\label{SecFormalApproach}

To analyze the singular perturbation problem at hand, the usual approach relies on the WKB change of unknown (\cite{GB.BP:07, Di-Ja-Mi-Pe}), which consists in the change of variable  $m_\eps(t,x,y)=e^{\frac{v_\eps(t,x,y)}{\eps}}$. 
In the context of concentrations, this form is motivated by the heuristics that a  Dirac mass is nothing but a narrow Gaussian. Indeed, in a weak sense 
$\left(\pi\eps\right)^{\frac{-n}{2}} e^{-\frac{\Vert y-\bar y\Vert ^2}{\eps}} \rightharpoonup \delta_{\bar y=y}$ as $\epsilon\to0$.
This approach has been extensively used in works on a similar issue, see for instance \cite{BG-LCE-PES,combustion,Champagnat2011a, Evans1989}. With this change of unknown, at the limit $\eps\to0$, the function $v_\eps(t,x,y)$ satisfies a constrained Hamilton-Jacobi equation, on which estimates can be difficult to prove because it carries all the difficulties in the asymptotic analysis, concentration effect and profile defined by the corrector. For that reason, the perturbed test function method has been invented \cite{Evans_perturbed} and widely used,  which avoids computing the corrector. 

Here, we propose a variant of the method. The principle can be viewed as a Taylor expansion $v_\eps(t,x,y)= v^1_\eps(t,y)+\eps v_\eps^2(t,x,y)$, to choose only some convenient terms to define the Hamilton-Jacobi equation for $v_\eps^1$, and then to prove that the corrector $v^2_\eps$ is bounded.
With a slight rewriting, we proceed to the change of variable
\begin{equation}\label{factorization}
m_\eps(t,x,y)=p_\eps(t,x,y) \, \exp{\frac{U_\eps(t,y)- \int_0^t \rho_\eps(s)ds }{\eps}}
\end{equation}
where $U_\eps(t,y)$ is defined \emph{ad hoc} through a standard Hamilton-Jacobi equation for which classical regularity properties can be proved. Then,the new and difficult step is to prove estimates on the corrector $p_\eps(t,x,y)$. Note that $p_\eps$ satisfies a linear equation rather than a constrained Hamilton-Jacobi equation: it makes thus possible to use  classical comparison principles and ideas issued from the General Relative Entropy method (see \cite{PMSMBP1}).
\\

We are now left with the task of finding a good candidate for $U_\eps(t,y)$ and formally identifying $p_\eps(t,x,y)$. Injecting \eqref{factorization} in \eqref{equa}, we find
\begin{equation}\label{equaFactor}
\left\{\begin{array}{ll}
\eps \D_t p_\eps + \D_x\left[A(x,y)p_\eps\right] +d(x,y)p_\eps+ \D_t U_\eps(t,y) p_\eps=0,
\\[2mm]
A(0,y)p_\eps\left(t,0,y\right)=\frac{1}{\eps^n}\int_{\R^n}\int_{\R_+}{M(\frac{y'-y}{\eps})e^{\frac{U_\eps(t,y')-U_\eps(t,y)}{\eps}}b(x',y')p_\eps(t,x',y')dx'dy'}.
\end{array}\right. 
\end{equation}
With the change of variable $z=\frac{y'-y}{\eps}$, the renewal term becomes
\begin{equation}\label{boundarycondeps}
A(0,y)p_\eps(t,0,y)=\int_{\R^n}\int_{\R_+}{M(z)e^{\frac{U_\eps(t,y+\eps z)-U_\eps(t,y)}{\eps}}b(x',y+\eps z)p_\eps(t,x',y+\eps z)dx'dz}.
\end{equation}
When $\eps$ is small, we can formally approximate
\begin{equation}\label{Approximation}
A(0,y)p_\eps(t,0,y)\approx\eta_\eps(t,y)\int_{\R_+}b(x',y)p_\eps(t,x',y)dx',
\end{equation}
where
\begin{equation}
\eta_\eps(t,y):=\int_{\R^n}{M(z)e^{\frac{U_\eps(t,y+\eps z)-U_\eps(t,y)}{\eps}}dz}.
\end{equation}
Then, formally putting $\eps\D_tp_\eps =O(\eps)$ in the first line of~\eqref{equaFactor}, we end up with the following approximate problem
\begin{equation}\label{equaFactorApprox}
\left\{\begin{array}{ll}
\D_x\left[A(x,y)p_\eps\right] +d(x,y)p_\eps + \D_t U_\eps(t,y) p_\eps=O(\eps),
\\[2mm]
A(0,y)p_\eps\left(t,0,y\right)=\eta_\eps(t,y)\int_{\R_+}b(x',y)p(t,x',y)dx' + O(\eps).
\end{array}\right. 
\end{equation}

Considering $\eta_\eps(t,y)$ as a parameter, we introduce the following eigenproblem: for fixed $(y,\eta)\in\R^n\times(0,+\infty)$, find $(\Lambda(y,\eta),Q(x,y,\eta))$, solution of
\begin{equation}\label{EigenProblemFinal}
\left\{\begin{array}{ll}
\D_x\left[A(x,y)Q\right]+d(x,y)Q-\Lambda(y,\eta)Q=0,\quad\forall x>0,\\[1mm]
A(0,y)Q(0,y,\eta)=\eta,\\[1mm]
Q> 0, \quad \int_{\R_+} b(x,y) Q(x,y,\eta)dx=1.
\end{array}\right.
\end{equation}
The third line corresponds to a normalization of the eigenfunction which is convenient for future calculations. Formally, $\Lambda$ corresponds to the \emph{effective fitness}, and $Q$ to the age profile at equilibrium in an environment characterized by the parameters $(y,\eta)$. 
\\

In an attempt to indentify $p_\eps$ with $Q$, this formal approach suggests to define $U_\eps$ as a solution of the Hamilton-Jacobi equation 
\begin{equation}\label{DefU_eps}\left\{
\begin{aligned}
&\D_t U_\eps(t,y)=-\Lambda\left(y,\int_{\R^n}M(z)e^{\frac{U_\eps(t,y+\eps z)-U_\eps(t,y)}{\eps}}dz\right) &\quad \forall t\geq0, \,  y\in\R^n, 
\\
&U_\eps(0,y)=U_\eps^0(y) &\forall y\in\R^n,
\end{aligned}\right.
\end{equation}
for some initial conditions $U^0_\eps$.

We stress out that, when $\eps$ vanishes, the full term $\exp{\frac{U_\eps(t,y)- \int_0^t \rho_\eps(s)ds }{\eps}}$  represents a bounded measure, e.g., a Dirac mass as in the example of Gaussian concentration,  and thus 
$$
\sup_{y \in \R^n} U_\eps(t,y)- \int_0^t \rho_\eps(s)ds =0, \qquad  \forall t\geq 0, 
$$
which explains the first formula in Theorem~\ref{MainResults}.

However, because of the non-local term $\eta_\eps$, proving uniform in $\eps>0$ estimates on $U_\eps$  is quite technical, a fact that can also be seen on the limiting equation when $\eps\to0$.
Taking for granted that $U_\eps$ converges to a function $U$ locally uniformly, $U$ turns out to be a viscosity solution of the Hamilton-Jacobi equation
\begin{equation}\label{DefU}\left\{
\begin{aligned}
&\D_t U(t,y)= H(y,\nabla_y U) &\forall t\geq0,\ \forall y\in\R^n,\\
&U(0,y)=U^0(y) &\forall y\in\R^n.
\end{aligned}\right.
\end{equation}
with the Hamiltonian $H(y,p)$ defined by 
\begin{equation}\label{DefHamiltonian}
H(y,p):= -\Lambda\left(y,\int_{\R^n}M(z)e^{p\cdot z}dz\right).
\end{equation}
From this equation, it is classical to prove uniform a priori bounds on $\D_t U$ (indeed, $\D_t U$ satisfies a transport equation). We deduce that $H(y,\nabla_y U)$ is bounded, then that 
\begin{equation}
\eta(t,y):= \int_{\R^n}M(z)e^{\nabla_y U(t,y)\cdot z}dz
 \end{equation} 
is bounded.
Besides, since $p\mapsto H(y,p)$ is convex, we deduce that $U$ is semi-convex, namely in space-time, the Hessian $D^2 U$ is bounded from below. Since $\D_t U$ is bounded, we have $\D^2_t U \in L_{loc}^1$. Then, using $\D^2_t U= -\D_t \eta \, \D_\eta\Lambda$, we infer that $\D_t\eta\in L^1_{loc}$. In the sequel, all these estimates are proved on $U_\eps$, uniformly in $\eps>0$. These are the classical (and sharp) general estimates for Hamilton-Jacobi equations with convex Hamiltonians. 
\\

With these optimal estimates on $U_\eps$ in hands, we can bound the corrector $p_\eps$. 
We set  
$$
Q_\eps(t,x,y):= Q(x,y,\eta_\eps(t,y)), \qquad \Lambda_\eps(t,y):=\Lambda(y,\eta_\eps(t,y))
$$ 
and find that 
\begin{equation}\label{EquationOnQepsFormelle}
\left\{\begin{array}{ll}
\eps \D_t Q_\eps + \D_x\left[A(x,y)Q_\eps\right] +d(x,y)Q_\eps+ \D_t U_\eps(t,y) Q_\eps=\eps\D_t Q_\eps,
\\[2mm]
A(0,y)Q_\eps\left(t,0,y\right)=\frac{1}{\eps^n}\int_{\R^n}\int_{\R_+}{M(\frac{y'-y}{\eps})e^{\frac{U_\eps(t,y')-U_\eps(t,y)}{\eps}}b(x',y')Q_\eps(t,x',y')dx'dy'}.
\end{array}\right. 
\end{equation}
Note that the boundary term at $x=0$ is obtained using the definition of $\eta_\eps(t,y)$ and the normalization $\int_{\R_+} b(x,y) Q(x,y,\eta)dx=1$ for all $(y, \eta)$.

The right hand side of the first line can be controlled with the available a priori bounds on $U_\eps$. Indeed, integrating equation~\eqref{EigenProblemFinal} to obtain $AQ$, we compute 
\begin{equation}
\D_t Q_\eps = \D_t\eta_\eps(t,y) \D_\eta Q(x,y,\eta_\eps(t,y))=\D_t\eta_\eps\left(\frac{1}{\eta_\eps}+\D_\eta\Lambda(y,\eta_\eps)\int_0^x\frac{1}{A(\cdot,y)}\right)Q_\eps.
\end{equation}
Except from this term, we see that $Q_\eps$ and $p_\eps$ satisfy the same equation~\eqref{equaFactor}, which is linear and admits a comparison principle. 
Under assumptions on the initial conditions, we deduce that $p_\eps$ is bounded from above and below by multiples of $Q_\eps$. Passing to the limit in the equation, we prove that $p_\eps$ converges weakly to a multiple of $Q$. It justifies our approach, especially the approximation~\eqref{Approximation}, and formally proves~\autoref{th:MainResultsConvergenceP}.
\\

Now, having in hand that $U_\eps$ converges (\autoref{th:MainResultsConvergenceU}) and that $p_\eps$ is uniformly bounded (\autoref{th:MainResultsConvergenceP}), \autoref{MainResults} can be understood and formally justified as follows.
On the one hand, the saturation term "$\rho_\eps(t)$" in \eqref{equa} implies the total population $\rho_\eps$ to be bounded, uniformly in $\eps>0$.
On the other hand, from \eqref{factorization}, the asymptotics of $m_\eps(t,\cdot,\cdot)$ when $\eps$ vanishes are driven by the points $y$ where $U(t,\cdot)$ is maximal. In other words, when $\eps\to0$, $m_\eps$ vanishes outside the set $\mathcal{S}$ defined in~\eqref{DefSetS}. Then, we can study the evolutionary dynamics through the dynamics of the (unique) critical point of $U(t,\cdot)$.

\subsection{Outline of the paper}
Section~\ref{SecU} is devoted to the definition of the eigenelements $(Q,\Lambda)$, the definition of the ansatz $U_\eps$, the statement of a priori estimates, and the asymptotics of $U_\eps$ when $\eps\to0$. The proofs are postponed to section~\ref{sec:appendixConcentrationWithMutations}. In section~\ref{sec:StudyOfP}, we study the corrector $p_\eps$ and prove~\autoref{th:MainResultsConvergenceP}.
Next, we study the asymptotics of the population density, and we prove \autoref{MainResults} in section~\ref{sec:Mutation_Proof}. Finally, some longer or more technical proofs are gathered in section~\ref{sec:appendixConcentrationWithMutations}.

\subsection{Assumptions}\label{SecAssumptions}
Most of our assumptions are formulated directly on the solution $(\Lambda(y,\eta),Q(x,y,\eta))$ of the limiting eigenproblem~\eqref{EigenProblemFinal} and, therefore, may seem abstract to the reader. However, we think that, besides being quite general, this formulation gives a better insight into the nature of our assumptions and the spirit of our approach.

\subsubsection{Example}\label{RemarksOnAssumptions}

Before stating the general assumptions,  we first give for the reader's convenience a concrete set of assumptions on the coefficients $A$, $b$, $d$ which are sufficient to fulfil the general assumptions. Note that, besides, we need the initial conditions to be "well prepared," which is not detailed in this example.

To avoid any difficulty when $\vert y\vert\to+\infty$, we can assume for example that $A$, $b$ and $d$ have a compact dependence on $y$. Namely,
if $\psi$ is a globally smooth diffeomorphism from $\R^n$ into the unit ball, we assume $A(x,y):= A_\star(x,\psi(y))$, $b(x,y):=b_\star(x,\psi(y))$, $d(x,y):=d_\star(x,\psi(y))$, where $A_\star$, $b_\star$ and $d_\star$ are defined on the closed unit ball.
This way, the $y$ space $\R^n$ is compactified.
Then, the coefficients can be chosen to fullfill the following conditions, for all $x\geq0$ and uniformly in $y\in\overline{B}_1$,
\begin{equation}
\left\{\begin{gathered}
A_\star(\cdot,\cdot),\ b_\star(x,\cdot) \geq0,\ d_\star(x,\cdot) \geq0 \text{ are }C^1,\\
 A_\star(x,y)\geq\underline{A}>0, \qquad \frac{1}{A_\star(\cdot,y)}\;  \text{is integrable on the support of } b_\star(\cdot,y),\\
b_\star(x,y)\leq Ke^{Kx}, \qquad \underline\eta b_\star(x,y)-d_\star(x,y)\geq \underline{r},\\
M(\cdot)\text{ is a Gaussian probability kernel}, 
\end{gathered}\right.
\end{equation}
for some constants $K>0$, $\underline{r}>0$ and where $\underline \eta$ is determined from the initial conditions, see~\eqref{AssumptionFWith}. 

These assumptions can be substantially generalized. For instance, using the formula (obtained by integrating the first line in \eqref{EigenProblemFinal})
\begin{equation}
\Lambda(y,\eta)=\frac{\int_{\R_+}\left(d(x,y)-\eta b(x,y)\right) Q(x,y,\eta)dx}{\int_{\R_+}Q(x,y,\eta)dx},
\end{equation}
the relation between $b_\star$ and $d_\star$ is only used to ensure the inequality 
$$
\Lambda(y,\eta)\leq -\underline{r}<0,\quad \forall \eta\geq\underline{\eta}
$$ 
which is sufficient in the sequel, see~\eqref{AssumptionFWith}.

\subsubsection{Assumptions on the coefficients}
\label{section:AssCoef}
We assume that for some $\underline{A}>0$
\begin{equation}\label{AssumptionCoefPositifs}
b\geq0,\quad d\geq0,\quad A\geq\underline{A}>0\text{ are continuous functions},
\end{equation}
\begin{equation}
A(\cdot,\cdot),\ b(x,\cdot),\ d(x,\cdot)\text{ are }C^1,
\end{equation}
\begin{equation}\label{AssumptionANoDegenerate}
d(\cdot,y), \ A(\cdot,y) \text{ are bounded from above and below in some interval of $\R_+$, uniformly in $y$}.
\end{equation}
which can be viewed as a non-degeneracy condition. Regarding the mutation kernel, we assume that
\begin{equation}
 M(\cdot) \text{ is a probability kernel and vanishes faster than any exponential.}
 \end{equation}
For instance, $M(\cdot)$ can be a Gaussian distribution or have a compact support. Note that the case without mutations corresponds to $M=\delta_0$ and has been already treated in~\cite{Nordmann2018a}.

In addition, we assume
\begin{equation}
\exists K>0\text{ such that } \sup_{ y\in\R^n} \int_0^{\overline x}\frac{1}{A(x',y)}dx'\leq ,K\label{AssumptionIntegralACombined}
\end{equation}
where $\overline x$ defines the largest support of $b(x,y)$,
\begin{equation}
\overline x:=\sup\left\{x\geq0 : \exists y\in\R^n, \ b(x,y)>0\right\}\in[0,+\infty].
\end{equation}
This assumption means that the transport outwards the support of $b$ occurs in finite time (it can be seen on the characteristics).
It is a necessary and sufficient condition for the ratio $\frac{\D_\eta Q}{Q}$ to be bounded on $[0,\overline{x}]$, which is used in section~\ref{SecP} to prove an estimate on $p_\eps$. 
Our approach would also work if the support of $A(\cdot,y)$ is compact, but we omit this case for simplicity.

We also  need the eigenvalue $\Lambda$ of~\eqref{EigenProblemFinal} to be well defined and differentiable (w.r.t $\eta$). As we see it in \autoref{ThEigenElements},defining,
\begin{equation}\label{DefinitionF}
F(y,\lambda):=\int_{\R_+}\frac{b(x,y)}{A(x,y)}\mathrm{exp}\left(\int_0^x\frac{\lambda-d(x',y)}{A(x',y)}dx'\right)dx,  \qquad \lambda\in\R,
\end{equation}
the eigenvalue $\Lambda(y,\eta)$ is defined through the relation
$$
F(y,\Lambda(y,\eta))=\frac{1}{\eta}.
$$
We assume there exists $\overline \Lambda<0$ such that
\begin{equation}\label{AssumptionFormelle}
F( y,\overline\Lambda),\ \D_\lambda F( y,\overline\Lambda) <+\infty\quad\text{for all }y\in\R^n.
\end{equation}
This assumption is not very restrictive, it is satisfied if $b(x,y)\leq K'e^{Kx}$ choosing $ \overline{\Lambda } \leq -K$.

\subsubsection{Assumptions on the initial conditions.}
We need the population $m_\eps(t,x,y)$ to be "well prepared" for concentration. We write $m_\eps^0(x,y)=p_\eps^0(x,y)e^{\frac{U_\eps^0(y)}{\eps}}$ according to \eqref{factorization},  and we assume that
\begin{equation}
U_\eps^0 \text{ smoothly converges to a function }U^0\text{ when $\eps$ vanishes},
\end{equation}
\begin{equation}\label{initial_Lipschitz}
\exists k^0>0\text{ such that }\forall\eps>0,\ \forall y\in\R^n,\quad \vert\nabla_y U_\eps^0(y)\vert\leq k^0,
\end{equation}
\begin{equation}\label{AssumptionInitialSemiConvexity}
\exists C>0\text{ such that }\forall\eps>0,\ \forall y\in\R^n,\quad \D^2_{y_i} U_\eps^0(y)\geq -C,
\end{equation}
\begin{equation}\label{initial_integrability}
\underline{J}^0\leq\int_{\R^n}e^{\frac{U^0_\eps(y)}{\eps}}dy\leq\overline{J}^0,\text{ for some }\underline{J}^0,\overline{J}^0>0,
\end{equation}
and $p^0_\eps$ such that, for some $\underline{\gamma}^0, \overline{\gamma}^0>0$:
\begin{equation}
\underline{\gamma}^0\leq \frac{p^0_\eps(x,y)}{Q(x,y,\eta_\eps^0(y))}\leq \overline {\gamma}^0,\label{AdditionalAssumption2Pre}
\end{equation}
where $Q$ is defined through~\eqref{EigenProblemFinal} and we define 
\begin{equation}\label{DefinitionEtaLambda0}
\eta_\eps^0(y):=\int_{z\in\R^n}M(z)e^{\frac{U^0_\eps(y+\eps z)-U^0_\eps(y)}{\eps}}dz, \qquad \Lambda_\eps^0(y):=\Lambda(y,\eta_\eps^0(y)).
\end{equation}
Another condition is also required on $\eta_\eps^0$, see~\eqref{AssumptionFWith}. Note that~\eqref{initial_integrability} and \eqref{AdditionalAssumption2Pre} ensures that $\rho_\eps^0:=\iint_{\R_+\times\R^n} m_\eps^0$ is uniformly bounded. 
Note also that assumption~\eqref{initial_integrability} implies $\sup_{\R^n} U^0=0.$

\subsubsection{Assumptions on the distribution of phenotypes}

The following assumptions only deal with the dependence of the coefficients on $y\in\R^n$. In particular, if all the coefficients have a compact dependence on $y$, then all the following assumptions are automatically satisfied.

First, we need a condition on the initial data, namely that, from~\eqref{DefinitionEtaLambda0}, $\eta^0_\eps(y)$  is bounded and 
$\Lambda_\eps^0(y)$ is bounded and negative. More precisely, in accordance to section~\ref{section:AssCoef}, we assume that there are two negative constants $\underline\Lambda\leq\overline\Lambda<0$ and two positive constants $0<\underline\eta\leq\overline\eta$, such that for all $y\in\R^n$,
\begin{equation}\label{AssumptionFWith}
\frac{1}{\overline{\eta}}\leq F(y,\underline{\Lambda})\leq F(y,\Lambda_\eps^0(y)):=\frac{1}{\eta_\eps^0(y)} \leq F(y,\overline{\Lambda})\leq\frac{1}{\underline{\eta}}.
\end{equation}
This assumption implies $\underline{\eta}\leq\eta_\eps^0(y)\leq\overline\eta$ and $\underline{\Lambda}\leq \Lambda^0_\eps(y)\leq \overline{\Lambda}$ (since $\lambda\mapsto F(y,\lambda)$ is increasing). We will see in \autoref{CorolaireImportant} that those two inequalities hold for all times, namely $\underline{\eta}\leq\eta_\eps(t,y)\leq\overline\eta$ and $\underline{\Lambda}\leq \Lambda(y,\eta_\eps(t,y))\leq\overline{\Lambda}$.
Note that with the notations
\begin{equation}\label{DefinitionBornes}
\frac{1}{\underline\eta(y)}:=F(y,\overline\Lambda),\quad
\frac{1}{\overline\eta(y)}:=F(y,\underline\Lambda),
\end{equation}
assumption~\eqref{AssumptionFWith} can be written
\begin{equation}
\underline{\eta}\leq\underline\eta(y)\leq\eta_\eps^0(y)\leq\overline\eta(y)\leq\overline{\eta}.
\end{equation}
We stress out that this assumption implies $-\Lambda(y,\eta_\eps(t,y))>0$. It means that every phenotype has a positive fitness, and is thus able to survive in absence of other phenotypes. This assumption is somehow restrictive, but it is not irrealistic, and allows to avoid some technicalities.
\\

Finally, the next assumptions deal with the derivatives of $\Lambda$. We need $\D_\eta\Lambda$, $\nabla_y\Lambda$ and $\nabla_y\D_\eta\Lambda$ to be bounded, and $\Lambda$ to be semi-convex.
According to~\eqref{FormuleDeriveeLambda},
we assume that there exists two constants $l,L>0$ such that for all $y\in\R^n$, $\lambda\in[\underline{\Lambda},\overline{\Lambda}]$,
\begin{equation}\label{AssumptionDerF}
l\leq  \D_\lambda F(y,\lambda)\leq L,
\end{equation}
\begin{equation}\label{AssumptionDerFy2}
 \vert\nabla_yF(y,\lambda)\vert,\vert\nabla_y\D_\lambda F(y,\lambda)\vert\leq L,
\end{equation}
\begin{equation}\label{AssumptionSemiConvexity}
\forall i\in\{1,\dots,n\},\quad\D^2_{y_i} F(y,\lambda)\geq -L.
\end{equation}

\section{Definition and properties of $U_\eps$}
\label{SecU}
To make sense to the above heuristic, we first give a rigorous definition of the eigenelements $(\Lambda,Q)$, which only uses classical arguments. Then, we define  $U_\eps$, formally introduced in section~\ref{SecFormalApproach}, and state some a priori estimates. 
We use those results to derive estimates on $\eta_\eps$ and $\Lambda_\eps$.
Note that the study of $U_\eps$ is autonomous and can be done separately from the analysis of the corrector. Finally, we pass to the limit, as $\eps\to 0$, in the quantities $U_\eps$, $\eta_\eps$ and $\Lambda_\eps$ to recover a viscosity solution of the Hamilton-Jacobi equation.
The longer or more technical proofs are postponed to section~\ref{sec:appendixConcentrationWithMutations}.

\subsection{The eigenproblem and effective Hamiltonian}

We consider the limiting problem~\eqref{EigenProblemFinal} and prove the existence of the eigenelements $(\Lambda,Q)$, along with some properties. The proof only uses elementary arguments and is postponned to section~\ref{sec:AppendixEigenElements}.
\begin{proposition}\label{ThEigenElements}
Under the assumptions of section~\ref{SecAssumptions}, for fixed $y\in\R^n$ and $\eta\in (\underline\eta(y),\overline\eta(y))$ (from~\eqref{DefinitionBornes}), there exists a unique couple $(\Lambda(y,\eta),Q(x,y,\eta))$
which satisfies \eqref{EigenProblemFinal}. 

Moreover, with $F$ defined in \eqref{DefinitionF}, $\underline{\Lambda}$ and $\overline{\Lambda}$ defined in \eqref{AssumptionFWith}, the eigenvalue $\Lambda(y,\eta)$ is continuously differentiable and  it holds,  for all $y\in\R^n,\; \eta\in (\underline\eta(y),\overline\eta(y))$,
\begin{equation}\label{FormuleImplicite}
F(y,\Lambda(y,\eta))=\frac{1}{\eta},
\end{equation}
\begin{equation}\label{Prerequis:BoundsOnLambda}
\underline{\Lambda}\leq\Lambda(y,\eta)\leq\overline{\Lambda}<0,
\end{equation}
\begin{equation}\label{DefinitionQ}
        Q(x,y,\eta)= \eta\frac{1}{A(x,y)}\mathrm{exp}\left(\int_0^{x}{\frac{\Lambda(y,\eta)-d(x',y)}{A(x',y)}d x'}\right).
\end{equation}
\end{proposition}
\begin{proof}
See section~\ref{sec:AppendixEigenElements}.
\end{proof}
 Differentiating the relation~\eqref{FormuleImplicite}, we immediately obtain the derivatives of $\Lambda$
\begin{equation}\label{FormuleDeriveeLambda}
\D_\eta \Lambda(y,\eta)=\frac{-1}{\eta^2\D_\lambda F(y,\Lambda(y,\eta))}  \leq \frac{-1}{\eta^2L} ,\qquad 
\nabla_y \Lambda(y,\eta)=\frac{-\nabla_y F(y,\Lambda(y,\eta))}{\D_\lambda F(y,\Lambda(y,\eta))}.
\end{equation}
In particular, the property $\D_\eta\Lambda<0$,  turns out to be fundamental in the following section. 

We also have some kind of concavity property for $\Lambda$. 
For all $y\in\R^n$, $\eta\in (\underline{\eta}(y),\overline{\eta}(y))$, we have
\begin{equation}\label{IntermediateInequality} 
\D_\eta^2\Lambda(y,\eta)+\frac{\D_\eta\Lambda(y,\eta)}{\eta}\leq 0.
\end{equation}

\begin{proof}
We use the short notations   $F:=F(y,\Lambda(y,\eta))$ (where $F$ is defined in~\eqref{DefinitionF}) and $\Lambda:=\Lambda(y,\eta)$.
Differentiationg the first relation in ~\eqref{FormuleDeriveeLambda} with respect to $\eta$, we find
\begin{equation}
\D_\lambda^2F (\D_\eta\Lambda)^2+\D_\lambda F\D_\eta^2 \Lambda=\frac{2}{\eta^3}, \qquad  \D_\eta^2 \Lambda=\frac{2}{\eta^3\D_\lambda F}-\frac{1}{\eta^4(\D_\lambda F)^3}\D_\lambda^2 F.
\end{equation}
Combining it with the expression of $\D_\eta \Lambda(y,\eta)$ in~\eqref{FormuleDeriveeLambda} and~\eqref{FormuleImplicite},
we deduce 
\begin{equation}
\D_\eta^2\Lambda+\frac{\D_\eta\Lambda}{\eta}= \frac{1}{(\D_\lambda F)^3\eta^3}\left((\D_\lambda F)^2-F\D_\lambda^2F\right).
\end{equation}
Using the Cauchy-Schwarz inequality and the definition of $F$ in~\eqref{DefinitionF}, we have
\begin{equation}\label{Inequality2}
(\D_\lambda F)^2-F\D_\lambda^2F\leq 0
\end{equation}
and the proof of \eqref{IntermediateInequality}  is completed.
\end{proof}

\subsection{Construction of $U_\eps$ and a priori estimates}
We give a rigourous definition of $U_\eps$, formally introduced in~\eqref{DefU_eps}. 
\begin{proposition}\label{TheoremUeps}
Under the assumptions of section~\ref{SecAssumptions}, for all $\eps>0$ there exists a solution $U_\eps(t,y)$ of~\eqref{DefU_eps}. In addition, it satisfies, with $\underline{\Lambda},\overline{\Lambda}<0$ defined in~\eqref{AssumptionFWith},
\begin{equation}
-\overline{\Lambda}\leq \D_t U_\eps(t,y)\leq-\underline{\Lambda}, \qquad \forall  \eps>0, \; t\geq0\;  y\in\R^n.
\end{equation}

\end{proposition}
\begin{proof}
See section~\ref{sec:ConstructionU}.
\end{proof}
In fact, $U_\eps$ is the unique solution of~\eqref{DefU_eps} with a locally bounded time derivative. 

As a direct consequence of \autoref{TheoremUeps}, we deduce the following useful corollary.
\begin{corollary}\label{CorolaireImportant}
With $\underline{\Lambda},$ $\overline{\Lambda}$, $\underline{\eta}$, $\overline{\eta}$ defined in \eqref{AssumptionFWith}, and setting 
\begin{equation}\label{DefinitionEtaLambda}
\eta_\eps(t,y):=\int_{\R^n}M(z)e^{\frac{U_\eps(t,y+\eps z)-U_\eps(t,y)}{\eps}}dz, \qquad 
\Lambda_\eps(t,y):=\Lambda(y,\eta_\eps(t,y)),
\end{equation}
we have 
\begin{equation}\label{BoundOnEta}
\underline\eta\leq\eta_\eps(t,y)\leq\overline{\eta} ,
\qquad 
 \underline\Lambda \leq\Lambda_\eps(t,y)\leq\overline\Lambda<0 .
\end{equation}
\end{corollary}
\begin{proof}
Simply use \autoref{TheoremUeps}, $\Lambda_\eps=-\D_t U_\eps$ and assumption~\eqref{AssumptionFWith}.
\end{proof}


\subsection{Further estimates}

\begin{proposition}\label{UepsLipschitz}
Under the assumptions of section~\ref{SecAssumptions}, with $k^0$ defined in~\eqref{initial_Lipschitz}, $L,l>0$ in~\eqref{AssumptionDerF} and $\underline{\eta}$ in~\eqref{BoundOnEta}, we have
\begin{equation}
\vert \nabla_y U_\eps(t,y)\vert \leq k^0+ \frac{L}{l\underline{\eta}^2}t,  \qquad \forall  \eps>0, \; t\geq0, \;  y\in\R^n.
\end{equation}
\end{proposition}
\begin{proof}
See section~\ref{sec:LipschitzEstimateEps}.
\end{proof}
Note that the coefficient $\frac{L}{l\underline{\eta}^2}$ comes from a bound on $\vert \nabla_y\Lambda (y,\eta_\eps)\vert$. We will see that, at the limit when $\eps\to0$, we can prove Lipschitz continuity globally in time.

We also need the following control of second order derivatives.
\begin{proposition}[Semi-convexity] \label{ThmSemiConvexity}
The function $U_\eps$ is semi-convex in $(t,y)$, that is, for all the $\nu\in\mathbb{S}:=\{(t,y)\in\R^{n+1}: t^2+\vert y\vert^2=1\}$, 
 $\D^2_{\nu \nu}U_\eps$ are bounded from below, uniformly in $\eps>0$, $y\in\R^n$, locally uniformly in $t\geq0$.
  Therefore  $U_\eps$ belongs to $W^{2,1}_{loc}$ in $(t,y)$, uniformly in $\eps>0$. 
\end{proposition}
\begin{proof}
The idea is to use that the Hamiltonian has properties closely related to convexity, namely~\eqref{IntermediateInequality} and to use the Lipschitz bounds. See section~\ref{sec:W11estimate}.
\end{proof}
The following corollary is essential when studying the corrector in section~\ref{sec:StudyOfP}.
\begin{corollary}\label{PropoW1estimate}
We have $\eta_\eps\in W^{1,1}_{loc}$, uniformly in $\eps>0$. In addition, denoting $\overline \eta(t):=\sup\limits_{y\in\R^n}\eta_\eps(t,y)$, we have
\begin{equation}
\int_0^T \left\vert \D_t\overline \eta_\eps(t) \right\vert dt\text{ is bounded uniformly in $\eps>0$}, \qquad \forall T\geq0.
\end{equation} 
\end{corollary}
\begin{proof}
See section~\ref{sec:ProofCorollary}.
\end{proof}

\subsection{Asymptotics}

With these regularity properties, we are ready to establish the asymptotics of $U_\eps$ when $\eps$ vanishes.
\begin{proposition}\label{TheoremU} 
Under the assumptions of section~\ref{SecAssumptions}, when $\eps$ vanishes, $U_\eps$ converges locally uniformly (and in $W_{loc}^{1,r}$, $1 \leq r < \infty$) to a function $U(t,y)\in W^{1,\infty}_{loc}$ which is a semi-convex  viscosity solution of
\begin{equation}\label{HJU}
\left\{\begin{aligned}
&\D_t U(t,y)=-\Lambda\left(y,\int_{\R^n}M(z)e^{\nabla_y U(t,y)\cdot z}dz\right),&\forall t>0,\ \forall y\in\R^n,\\
&U(0,y)=U^0(y), &\forall y\in\R^n.
\end{aligned}\right.
\end{equation}
Moreover, under the assumption that $M$ is not degenerate (i.e., $M(\cdot)>0$ in a neighborhood of $0$), the function $U(t,y)$ is globally Lipschitzian.
\end{proposition}
\begin{proof}
See section~\ref{sec:Asymptotics_of_U} and section~\ref{sec:APosterioriLipschitz}.
\end{proof}
We also point out that, from Proposition~4.3 in~\cite{Nordmann2018a},
\begin{equation}
p\mapsto-\Lambda\left(y,\int_{\R^n}M(z)e^{p\cdot z}dz\right)\text{ is a convex mapping, }\forall y\in\R^n.
\end{equation}
This class of Hamiltonian has been widely studied, and numerous results on regularity as well as representation formula are available~\cite{Bianchini2012,Fleming1975}. 

As a direct consequence of the $L^r_{loc}$ convergence of $\nabla U_\eps$ to $\nabla U$, we have the following corollary.
\begin{corollary}\label{CoroStrongCVEta}
When $\eps\to0$, $\eta_\eps$ converges in $L^r_{loc}$, $1 \leq r < \infty$,  to
\begin{equation}
\eta(t,y):=\int_{\R^n}M(z)e^{\nabla_yU\cdot z}dz.
\end{equation}
Consequently, $\Lambda_\eps(t,y):=\Lambda(t,\eta_\eps(t,y))$ converges to $\Lambda(y,\eta(t,y))$ and $Q_\eps(t,x,y):=Q(x,y,\eta_\eps(t,y))$ to $Q(x,y,\eta(t,y))$ in $L^r_{loc}$, $1 \leq r < \infty$.
\end{corollary}

\section{Asymptotics of the corrector - Proof of \autoref{th:MainResultsConvergenceP}}
\label{sec:StudyOfP}

We now turn to the main new results of the paper. These are boundedness from above and below and the asymptotics of the corrector $p_\eps(t,x,y)$ defined through the factorisation~\eqref{factorization}, according to the definition of $U_\eps$ in \eqref{DefU_eps}.

\subsection{Estimates on $p_\eps$}\label{SecP}
This section is devoted to the proof of the first statement of~\autoref{th:MainResultsConvergenceP}, that is, $p_\eps(t,x,y)$ and $\int_{x>0} p_\eps(t,x,y)$ are bounded uniformly in $\eps>0$, $x\geq0$, $y\in\R^n$, locally uniformly in $t\geq0$.

Our first result states a control of $p_\eps$, from above and below, for $x\in[0,\overline{x}]$ (with $\overline x$ from~\eqref{AssumptionIntegralACombined}).
\begin{lemma}\label{LemmaStrongBoundsOnP}
Under the assumptions of section~\ref{SecAssumptions}, for any fixed $T>0$, there exists two constants $\underline{\gamma},\overline{\gamma}>0$ such that
\begin{equation}
\underline{\gamma}\ Q_\eps(t,x,y)\leq p_\eps(t,x,y)\leq \overline{\gamma}\ Q_\eps(t,x,y),
\end{equation}
for all $\eps>0$, $t\in[0,T]$, $y\in\R^n$, $x\in [0,\overline{x}]$, where
\begin{equation}
Q_\eps(t,x,y):=Q(x,y,\eta_\eps(t,y)).
\end{equation}
\end{lemma}
\begin{proof}
The function $p_\eps(t,x,y)$ satisfies the following equation, for $\eps>0,\ t>0,\ x>0,\ y\in\R^n$,
\begin{equation}\label{EquationOnp}
 \left\{\begin{aligned}
 &\eps \D_tp_\eps+\D_x\left[A(x,y)p_\eps\right]+\left(d(x,y)-\Lambda_\eps\right)p_\eps=0,\\
 &A(0,y)p_\eps(t,0,y)=\iint\limits_{\substack{x>0,\ z\in\R^n}}M(z)e^{\frac{U_\eps(t,y+\eps z)-U_\eps(t,y)}{\epsilon}}b(x,y+\eps z)p_\eps(t,x,y+\eps z)dx dz.
 \end{aligned}\right.
\end{equation}
and $Q_\eps$ satisfies, for $(t, y)$ as parameters, the equation in the variable $x$
\begin{equation}
 \left\{\begin{aligned}
&\D_x[A(x,y)Q_\eps]-(d(x,y)-\Lambda_\eps (t,y) )Q_\eps= 0,\\
&A(0,y)Q_\eps(t,0,y)=\eta_\eps(t,y).
\end{aligned}\right.
\end{equation}
Setting
\begin{equation}
\gamma_\eps(t,x,y):=\frac{p_\eps(t,x,y)}{Q_\eps(t,x,y)},
\end{equation}
we have
\begin{equation}\label{EquationOnV}
 \left\{\begin{aligned}
&\D_t \gamma_\eps +\frac{A(x,y)}{\eps}\D_x \gamma_\eps= -\frac{\D_tQ_\eps}{Q_\eps} \gamma_\eps,\\
&\gamma_\eps(t,0,y)=\iint\limits_{\substack{x>0,\ z\in\R^n}} J_\eps(t,x,y,z) \gamma_\eps(t,x,y+\eps z)dz,
\end{aligned}\right.
\end{equation}
where
\begin{equation}\label{DefinitionProbaKernelJ}
J_\eps(t,x,y,z):= \frac{1}{\eta_\eps}
M(z)e^{\frac{U_\eps(t,y+\eps z)-U_\eps(t,y)}{\epsilon}}b(x,y+\eps z)Q_\eps(t,x,y+\eps z).
\end{equation}
Our goal is to infer some bounds on $\gamma_\eps$.

First, from the definition of $\eta_\eps$ and the normalization $\int_{x>0}b(x,y)Q(x,y,\eta)dx=1$,
we see that  $J_\eps$ is a probability kernel, 
$$
\iint_{\substack{x>0,\ z\in\R^n}} J_\eps(t,x,y,z) dxdz=1, \qquad \forall t \geq0, \; y \in \R^n.
$$

We need to estimate $ \frac{\D_tQ_\eps}{Q_\eps}.$ We compute, using the representation formula~\eqref{DefinitionQ}, 
\begin{align*}
 \frac{\D_t Q_\eps}{Q_\eps}(t,x,y)
&=  \D_t\eta_\eps(t,y) \frac{\D_\eta Q}{Q}(x,y,\eta_\eps(t,y))\\
&= \D_t\eta_\eps(t,y)\left(\frac{1}{\eta_\eps(t,y)}+\D_\eta\Lambda (y,\eta_\eps(t,y))\int_0^x \frac{1}{A(x',y)}dx'\right).
\end{align*}
Because  $\eta_\eps$ is bounded from below, see~\eqref{BoundOnEta}, and $\D_\eta\Lambda$ is bounded, see~\eqref{FormuleDeriveeLambda} and assumption~\eqref{AssumptionDerF}, we have
\begin{equation}\label{StepIntermediaire}
\left\vert \frac{\D_t Q_\eps}{Q_\eps}(t,x,y)\right\vert \leq  K\vert \D_t\eta_\eps(t,y)\vert\left(1+\int_0^x \frac{1}{A(x',y)}dx'\right),
\end{equation}
for some constant $K>0$. 
Then, using assumption~\eqref{AssumptionIntegralACombined}, we have, for $x\in[0,\overline{x}]$,
\begin{equation}
\left\vert \frac{\D_t Q_\eps}{Q_\eps}(t,x,y)\right\vert \leq  K\vert \D_t\eta_\eps(t,y)\vert
\end{equation}
for some constant still denoted by $K$.
Setting $\overline \eta_\eps(t):=\sup\limits_{y\in\R^n}\eta_\eps(t,y)$, we define
\begin{equation}
\underline  \gamma_\eps (t,x,y):= \gamma_\eps (t,x,y) \exp\left(-K\int_0^t\vert \D_t\overline\eta_\eps(t')\vert dt'\right),
\end{equation}
so  that $\underline  \gamma_\eps$ is a subsolution to~\eqref{EquationOnV}, namely
\begin{equation}
 \left\{\begin{aligned}
&\D_t \underline  \gamma_\eps +\frac{A(x,y)}{\eps}\D_x\underline \gamma_\eps\leq 0\\
&\underline \gamma_\eps(t,0,y)\leq \iint\limits_{\substack{x>0,\ z\in\R^n}} J_\eps(t,x,y,z) \underline \gamma_\eps(t,x,y+\eps z)dzdx
\end{aligned}\right.
\end{equation}
(in fact, equality holds in the second line).
From the comparison principle, we deduce
\begin{equation}
\underline \gamma_\eps (t,x,y)\leq \sup_{\substack{x\in[0,\overline x]\\ y\in\R^n}} \underline \gamma_\eps (0,x,y)\leq \overline{\gamma}^0,
\end{equation}
where $\overline \gamma^0$ comes from assumption~\eqref{AdditionalAssumption2Pre}. This gives a control  from above on $ \gamma_\eps (t,x,y) $ which  implies
\begin{equation}
p_\eps(t,x,y)\leq \overline{\gamma}^0 Q_\eps(t,x,y)\exp\left(K\int_0^t\vert \overline \D_t\eta_\eps(t')\vert dt'\right).
\end{equation}
By \autoref{PropoW1estimate}, $\int_0^t \vert \D_t\overline \eta_\eps\vert$ is bounded uniformly in $\eps>0$, therefore, for some constant~$\overline{\gamma}$,
\begin{equation}
p_\eps(t,x,y)\leq \overline{\gamma} Q_\eps(t,x,y).
\end{equation}

Identically, we infer the bound from below, and the proof of~\autoref{LemmaStrongBoundsOnP} is completed.
\end{proof}

\bigskip

With the previous lemma in hand, we now estimate $p_\eps$ for all $x\geq0$ (which is useless if $\overline x=+\infty$). We set
\begin{gather*}
\overline{Q}(x,y):=\overline{\gamma}\frac{\overline\eta}{\underline{\eta}(y)}Q(x,y,\underline{\eta}(y))=\overline{\gamma}\frac{\overline{\eta}}{A(x,y)}\exp\left(\int_0^x\frac{\overline{\Lambda}-d(x',y)}{A(x',y)}dx'\right),\\
\underline{Q}(x,y):=\underline{\gamma}\frac{\underline\eta}{\overline{\eta}(y)}Q(x,y,\overline{\eta}(y))=\underline\gamma\frac{\underline{\eta}}{A(x,y)}\exp\left(\int_0^x\frac{\underline{\Lambda}-d(x',y)}{A(x',y)}dx'\right),
\end{gather*}
where $\underline{\gamma},\overline{\gamma}$ are given by the previous lemma.
\begin{lemma}\label{LemmaBoundOnP}
Under the same condition as in the previous lemma, we have
\begin{equation}
\underline{Q}(x,y)\leq p_\eps(t,x,y)\leq\overline{Q}(x,y), \qquad \forall \eps>0, \; t\in[0,T], \; x\geq 0, \; y\in\R^n.
\end{equation}

\end{lemma}
Note that, when restricted to $[0, \overline x]$  these bounds are weaker than in  \autoref{LemmaStrongBoundsOnP} since $\frac{1}{\underline{\gamma}}\underline {Q}\leq Q_\eps \leq \frac{1}{\overline{\gamma}}\overline{Q}$. 
\begin{proof}
From \autoref{LemmaStrongBoundsOnP}, we deduce
\begin{equation}
\underline{\gamma}\ \underline{\eta}\leq A(0,y)p_\eps(t,0,y)\leq \overline{\gamma}\ \overline \eta.
\end{equation}
Hence, on the one hand, we have
\begin{equation}\label{EquationOnpBIS}
 \left\{\begin{aligned}
 &\eps \D_tp_\eps+\D_x\left[A(x,y)p_\eps\right]+\left(d(x,y)-\overline\Lambda\right)p_\eps\leq 0,\\
 &A(0,y)p_\eps(t,0,y)\leq  \overline{\gamma}\ \overline{\eta},
 \end{aligned}\right.
\end{equation}
for $\eps>0,\ t\in[0,T],\ x>0,\ y\in\R^n$.
On the other hand,
\begin{equation}
 \left\{\begin{aligned}
 &\eps \D_t{\overline Q}+\D_x\left[A(x,y){\overline Q}\right]+\left(d(x,y)-\overline\Lambda\right){\overline Q}= 0,\\
 &A(0,y){\overline Q}(0,y)=  \overline{\gamma}\ \overline{\eta}.
 \end{aligned}\right.
\end{equation}
From the comparison principle, we deduce $p_\eps\leq {\overline Q}$. The lower bound can be proved similarily.
\end{proof}

We are now ready to prove our first main result 
\begin{proposition}[Uniform estimates on $p_\eps$ in $L^1\cap L^\infty$] \label{thmCVWith} Under the assumptions of section~\ref{SecAssumptions} and for any fixed $T>0$,
$p_\eps(t,x,y)$ and $\int_{x>0}p_\eps(t,x,y)dx$ are bounded from above and below,  uniformly in $\eps>0$, $t\in[0,T]$, $x\geq0$, $y\in\R^n$.
\end{proposition}

\begin{proof}
The first point is deduced from \autoref{LemmaBoundOnP} and $\overline{Q}\leq \overline{\gamma}\frac{\overline{\eta}}{\underline{A}}$, with $\underline{A}$ from assumption~\eqref{AssumptionCoefPositifs}.

To prove the second point, we only need to estimate the integrals of $\underline{Q}$ and $\overline{Q}$. Recalling $\overline\Lambda<0$, we compute
\begin{align*}
\int_0^{+\infty} \overline{Q}(x,y)dx
&\leq \int_0^{+\infty}\overline{\gamma} \frac{\overline{\eta}}{A(x,y)}\exp\left(\int_0^x\frac{\overline{\Lambda}}{A(x',y)}dx'\right)\\
&=\left[\overline{\gamma}\frac{\overline{\eta}}{\overline{\Lambda}}\exp\left(\int_0^x\frac{\overline{\Lambda}}{A(x',y)}dx'\right)\right]_{x=0}^{+\infty}\\
&=\overline{\gamma}\frac{\overline\eta}{-\overline \Lambda}\left[1-\exp\left(\int_0^{+\infty}\frac{\overline{\Lambda}}{A(x',y)}dx'\right)\right]
\leq \overline{\gamma}\frac{\overline \eta}{-\overline \Lambda},
\end{align*}
which proves the bound from above.

For the other inequality, we use the non-degeneracy assumption~\eqref{AssumptionANoDegenerate} which implies
$\int_{\R_+}Q(x,y,\eta)dx>\alpha$, for some $\alpha>0$.
\end{proof}

\subsection{Asymptotics of $p_\eps$}\label{sec:Asymptotics_of_P}

We complete the proof of~\autoref{th:MainResultsConvergenceP} by showing that, when $\eps\to0$, the mutations affect the equilibrium distribution $Q(x, y, \eta)$ only through a multiplicative factor $\gamma(t,y)$,
\begin{proposition}[Convergence of the corrector] \label{PropoAsymptoticsP}
For $t\in[0,T]$ (with $T>0$ fixed), $x\in[0,\overline{x}]$, $y\in\R^n$, and up to extraction of a subsequence when $\eps\to0$, $p_\eps$ converges in $L^\infty$ weak-$^\star$  to $\gamma (t,y) Q(x,y,\eta(t,y))$ with $\gamma(t,y)\in L^\infty$ which formally satisfies the equation
\begin{equation}\label{EquationVFinale}
\left\{\begin{aligned}
&\D_t \gamma +\D_\eta\Lambda \left(\int_{z\in\R^n} M(z) e^{\nabla U\cdot z} z dz\right)\cdot \nabla_y \gamma +\D_t\eta \frac{\D^2_\eta\Lambda}{2\D_\eta\Lambda} \gamma=0,\\
&\gamma(t=0)=\gamma^0.
\end{aligned}\right.
\end{equation}
\end{proposition}
The difficulty in stating rigorously equation~\eqref{EquationVFinale} is that $\D_t\eta $ and $\int_{z\in\R^n} M(z) e^{\nabla U\cdot z} z dz$ are nothing more than bounded measures which is not smooth enough since $\gamma$ is just $L^\infty$.
Nevertheless, we can prove establish the convergence of the full sequence $p_\eps$ to $\gamma Q$ if $T$ is small enough: we use the regularity of the initial condition to rigourously pass to the limiting equation~\eqref{EquationVFinale}, and then use a standard uniqueness result on this equation.

Note also  that, according to \autoref{MainResults}, the population concentrates on $\mathcal{S}$ where $U(t,y)$ achieved its maximum. From \cite{GB.BP:07}, $U$ is differentiable at points of  $\mathcal{S}$ and one has  $\partial_t U= \nabla_y U = 0$ and $\eta =1$. Thus $p_\eps$ converges to a multiple of $Q(x,y,1)$ on $\mathcal{S}$.
In addition, if $M(\cdot)$ is even, then the drift term vanishes and the equation can be written
\[
\D_t [\D_\eta\Lambda(y, \eta(t,y))^{1/2} \gamma ]=0.
\]
\\

To  prove \autoref{PropoAsymptoticsP}, and establish~\eqref{EquationVFinale} we use the dual problem associated with~\eqref{EigenProblemFinal}; for fixed $(y,\eta)\in\R^n\times(0,+\infty)$ consider $\Phi(x,y,\eta)$ the unique solution of
\begin{equation}\label{EquationDualeBis}
\left\{\begin{aligned}
&A(x,y)\D_x\Phi+\left[\Lambda(y,\eta)-d(x,y)\right]\Phi=-\eta b(x,y)\Phi(0,y,\eta),\quad\forall x>0,\\
&\int_{x>0}Q(x,y,\eta)\Phi(x,y,\eta)dx=1.
\end{aligned}\right.
\end{equation}
We can solve this ordinary differential equation and  find
\begin{equation}\label{FormulePhi}
\Phi(x,y,\eta)=-\D_\eta\Lambda(y,\eta)\eta \int_x^{+\infty}\frac{b(x',y)}{A(x',y)}\exp\left(\int_x^{x'}\frac{\Lambda(y,\eta)-d(x'',y)}{A(x'',y)}d x''\right)dx'.
\end{equation}
Also, multiplying equation~\eqref{EquationDualeBis} by $\partial_\eta Q$ and integrating by parts, we get
\begin{equation}\label{FormuleIntegralePhi1}
\Phi(0,y,\eta)= -\D_\eta\Lambda(y,\eta),
\end{equation}
and, multiplying equation~\eqref{EquationDualeBis} by $\partial^2_{\eta \eta} Q$, we end up with
\begin{equation}\label{FormuleIntegralePhi2}
\int_0^{+\infty} Q(x,y,\eta)\D_\eta \Phi(x,y,\eta)dx=-\frac{\D^2_\eta\Lambda}{2\D_\eta\Lambda}.
\end{equation}

\begin{proof}
We fix $T>0$ and, throughout the proof, choose $t\in[0,T]$, $x\in[0,\overline{x}]$, $y\in\R^n$.
We define $\gamma_\eps:=\frac{p_\eps}{Q_\eps}$ and recall equation~\eqref{EquationOnV}. From \autoref{LemmaStrongBoundsOnP}, we know that $\gamma_\eps$ is bounded and thus converges (up to extraction of a subsequence) to some $\gamma$ in $L^\infty$ weak-$^\star$. Passing to the limit in~\eqref{EquationOnV}, we deduce $\D_x \gamma \equiv0$ (in the sense of distributions), and thus $\gamma$ only depends on $t$ and $y$. Since $Q_\eps$ (strongly) converges to some $Q$ (see \autoref{CoroStrongCVEta}), we deduce that $p_\eps(t,x,y)$ converges (up to extraction of a subsequence) to $\gamma(t,y)Q(t,x,y)$ in $L^\infty$ weak-$^\star$.

We are now left with the task of identifying $\gamma$. To do so, we set
\begin{equation}
E_\eps(t,y):= \int_0^{+\infty} \gamma_\eps(t,x,y)Q_\eps(t,x,y)\Phi_\eps(t,x,y) dx, \qquad 
\Phi_\eps(t,x,y):=\Phi(x,y,\eta_\eps(t,y)).
\end{equation}
 Note that $\Phi_\eps(t,x,y) \equiv 0$ for $x>\overline x$ and thus $E_\eps(t,y):= \int_0^{\overline x} \gamma_\eps Q_\eps\Phi_\eps$. 
From~\eqref{EigenProblemFinal}, \eqref{EquationOnV} and \eqref{EquationDualeBis}, we can write 
\begin{align*}
\eps\D_t E_\eps
&=\int_{x>0} \eps\D_t \gamma_\eps \;  Q_\eps\Phi_\eps +\eps \int_{x>0} \gamma_\eps  \D_t[Q_\eps\Phi_\eps]\\
&= -\int_{x>0} A\D_x \gamma_\eps \;  Q_\eps\Phi_\eps -\eps \int_{x>0} \gamma_\eps \D_t Q_\eps \Phi_\eps+\eps \int_{x>0} \gamma_\eps \D_t[Q_\eps\Phi_\eps]\\
&=[A \gamma_\eps Q_\eps\Phi_\eps](x=0) + \int_{x>0} \gamma_\eps \D_x[AQ_\eps\Phi_\eps]+\eps \int_{x>0} \gamma_\eps Q_\eps \D_t\Phi_\eps\\
\shortintertext{(from $\D_x[AQ_\eps\Phi_\eps]=-\eta_\eps b Q_\eps\Phi_\eps(x=0)$, and with the probability kernel $J_\eps$ defined in~\eqref{DefinitionProbaKernelJ})} 
&=  \eta_\eps\Phi_\eps(x=0) \iint\limits_{x>0,\ z\in\R^n} J_\eps(t,x,y,z)\left(\gamma_\eps(t,x,y+\eps z)-\gamma_\eps(t,x,y)\right)dzdx+\eps \int_{x>0} \gamma_\eps Q_\eps \D_t\Phi_\eps
\end{align*}
Recalling~\autoref{CoroStrongCVEta}, we know that $\Phi_\eps$ converges to some $\Phi$ when $\eps\to0$. Dividing by $\eps$ and passing to the limit $\eps\to0$ (after extracting a subsequence, in the sense of distributions), we deduce
\begin{equation}\label{EntropyLimite}
\D_t\left[ \gamma\int_{x>0} Q\Phi\right]=\D_t\gamma=\Phi(t,0,y)\left(\int_{z\in\R^n} M(z) e^{\nabla U\cdot z} z\right)\cdot \nabla_y \gamma+ \gamma\int_{x>0}Q\D_t\Phi.
\end{equation}

Injecting $\D_t\Phi_\eps(t,x,y) =\D_t\eta_\eps\D_\eta\Phi(x,y,\eta_\eps)$ and~\eqref{FormuleIntegralePhi1}-\eqref{FormuleIntegralePhi2} in \eqref{EntropyLimite}, we end up with equation~\eqref{EquationVFinale}. From classical uniqueness results, we deduce that $\gamma_\eps$ converges weakly to $\gamma$ for the whole sequence $\eps\to0$ (and not for an extracted subsequence).
\end{proof}

\section{Concentration of the population density - proof of \autoref{MainResults}}\label{sec:Mutation_Proof}

We now conclude on the consequences of our study on $U_\eps$ and $p_\eps$ with the concentration effect for the population density $m_\eps$ itself.

\subsection{Selection of the fittest phenotypes}\label{SecConcentration}
The following result states that the total population $\rho_\eps$ is uniformly bounded and converges when $\eps\to0$. Recalling \eqref{factorization}, the two first statements of \autoref{MainResults} are direct consequences of the following proposition and the uniform $L^1$ estimate on $p_\eps$ (\autoref{thmCVWith}).
\begin{proposition}\label{ThmConcentration}
Under the assumptions of section~\ref{SecAssumptions},
\\
1. There exist two positive constants $\underline{\rho},\overline{\rho}>0$ such that
\begin{equation*}
\underline\rho\leq\rho_\eps(t)\leq\overline{\rho},\quad\forall\eps>0,\ t\geq 0.
\end{equation*}
In addition, $\rho_\eps$ converges to some $\rho$ in the $L^\infty$-weak$^*$ topology.
\\
2. The integral $\int_{\R^n}e^{\frac{U_\eps(t,y)- \int_0^t\rho_\eps}{\eps}}dy$ is bounded away from $0$, uniformly in $\eps>0$, $t\geq0$. 

Consequently, when $\eps$ vanishes, we have
\begin{equation} \label{eqHJ_constrained}
\int_0^t\rho = \sup\limits_{y\in\R^n}U(t,y), \quad \forall t \geq 0 .
\end{equation}
\end{proposition}
In addition, with $\mathcal{S}$ defined in \eqref{DefSetS}, 
we have
\begin{align*}
\mathcal{S}
&:=\left\{t\geq0,\ y\in\R^n: U(t,y)=\sup_{\R^n} U(t,\cdot)\right\} =\left\{t\geq0,\ y\in\R^n: U(t,y)=\int_0^t\rho\right\} .
\end{align*}

\begin{remark}
T,he proof of~\autoref{ThmConcentration} becomes much simpler  if we assume that there are $\underline{r},\overline{r}>0$ such that 
\begin{equation}
\underline{r}\leq b(x,y)-d(x,y)\leq\overline{r},\quad \forall x\geq0,\ y\in\R^n,
\end{equation}
 Indeed, integrating~\eqref{equa} and using an integration by parts, we obtain
\begin{align*}
\eps\frac{\de}{\de t}\rho_\eps(t) &=\iint\limits_{\R^n\times\R_+}\left[\left(\frac{1}{\eps^n}\int_{\R^n}{M(\frac{y'-y}{\eps})d y}\right)b(x,y') -d(x,y')\right]m_\eps(t,x,y')dxdy' -\rho_\eps^2(t)\\
 &\displaystyle\leq \overline{r}\rho_\eps(t) -\rho_\eps^2(t). \phantom{\int}
\end{align*}
which implies
$0\leq\rho_\eps(t)\leq\max\left({\overline{r}},\rho_\eps^0\right)$ and provides us with an a priori upper bound on $\rho_\eps$. With the same method, we also infer a positive lower bound on $\rho_\eps$.

Then, using the uniform $L^1$ estimate on $p_\eps$ from \autoref{thmCVWith}, we directly deduce  \eqref{eqHJ_constrained}.
\end{remark}

\begin{proof}[of~\autoref{ThmConcentration}]
We recall that 
\begin{equation}\label{DefinitionRho}
\rho_\eps(t)=\iint\limits_{x,y}p_\eps(t,x,y)e^{\frac{U_\eps(t,y)-\int_0^t\rho_\eps}{\eps}}dxdy.
\end{equation}
Multiplying by $e^{\int_0^t \frac{\rho_\eps}{\eps}}$ we have
\begin{equation*}
\rho_\eps(t)e^{\frac{\int_0^t\rho_\eps}{\eps}}=\int_{\R^n} e^{\frac{U_\eps(t,y)}{\eps}}\int_{\R_+}p_\eps(t,x,y)dxdy,
\end{equation*}
and integrating over $(0,t)$, we deduce
\begin{equation*}
\eps\left(e^{\frac{\int_0^t\rho_\eps}{\eps}}-1\right)=\int_0^t\int_{\R^n}  e^{\frac{U_\eps(t,y)}{\eps}}\int_{\R_+} p_\eps(t,x,y)dxdydt.
\end{equation*}
From $0< -\overline{\Lambda}\leq \D_t U_\eps\leq -\underline{\Lambda}$ (\autoref{TheoremUeps}) and the $L^1(dx)$ estimate on $p_\eps$ (\autoref{thmCVWith}), we have
\begin{align*}
\eps\left(e^{\frac{\int_0^t\rho_\eps}{\eps}}-1\right)
&\geq \underline{I}\int_0^t\int_{\R^n} \frac{\eps}{-\underline{\Lambda}}\frac{\D_tU_\eps(t,y)}{\eps} e^{\frac{U_\eps(t,y)}{\eps}}dy\ dt\\
&\geq \frac{\eps \underline{I}}{-\underline\Lambda}\int_{\R^n}e^{\frac{U_\eps(t,y)}{\eps}}-e^\frac{U^0_\eps(y)}{\eps} dy.
\end{align*}
Dividing by $\eps e^{\frac{\int_0^t\rho_\eps}{\eps}}$ on both sides we find
\begin{equation*}
\left(1-e^{-\frac{\int_0^t\rho_\eps}{\eps}}\right)\geq \frac{ \underline{I}}{-\underline\Lambda}\int_{\R^n}\left(e^{\frac{U_\eps(t,y)}{\eps}}-e^\frac{U^0_\eps(y)}{\eps}  \right)e^{-\frac{\int_0^t\rho_\eps}{\eps}}dy,
\end{equation*}
that we rewrite, with $u_\eps(t,y) = U_\eps(t,y) - \int_0^t\rho_\eps $
\begin{equation}
\int_{\R^n}e^{\frac{u_\eps(t,y)}{\eps}}dy\leq \frac{-\underline{\Lambda}}{\underline I}\left(1-e^{-\frac{\int_0^t\rho_\eps}{\eps}}\right)+e^{-\frac{\int_0^t\rho_\eps}{\eps}}\int_{\R^n}e^\frac{U^0_\eps(y)}{\eps} dy.
\end{equation}
Then, from assumption~\eqref{initial_integrability}, we deduce
\begin{equation}
\int_{\R^n}e^{\frac{u_\eps(t,y)}{\eps}}dy\leq \frac{-\underline{\Lambda}}{\underline I}\left(1-e^{-\frac{\int_0^t\rho_\eps}{\eps}}\right)+\overline{J}^0e^{-\frac{\int_0^t\rho_\eps}{\eps}},
\end{equation}
and 
\begin{equation*}
\int_{\R^n}e^{\frac{u_\eps(t,y)}{\eps}}dy\leq \overline{K}:=\max\left(\frac{-\underline{\Lambda}}{\underline I}, \overline{J}^0\right).
\end{equation*}
Identically, we infer
\begin{equation}
\int_{\R^n}e^{\frac{u_\eps(t,y)}{\eps}}dy\geq \underline{K}:= \min\left(\frac{-\overline{\Lambda}}{\overline I},\underline{J}^0\right)>0.
\end{equation}

Now, from the definition of $\rho_\eps$ (recalled  in \eqref{DefinitionRho}) and the $L^1$ estimate on $p_\eps$ (\autoref{thmCVWith}) we find
\begin{equation*}
\underline I\ \underline K\leq \rho_\eps(t)\leq \overline I\ \overline K.
\end{equation*}

\bigskip

Since $\rho_\eps(t)$ is uniformly bounded, there exists a sequence $\eps_k\to0$ such that $\rho_{\eps_k}$ converges to some $\rho$ in the $\mathcal{L}^\infty$-weak$^*$ topology when $k\to+\infty$. Since, in addition, $U_\eps$ converges locally uniformly to $U$ (\autoref{TheoremU}), we deduce that $u_{\eps_k}$ converges locally uniformly to some $u$. 

Now, from $\underline{K}\leq\int_{\R^n}e^{\frac{u_\eps(t,\cdot)}{\eps}}\leq\overline{K}$, at the limit $k\to+\infty$, we have
\begin{equation}\label{LimsupNulle}
\forall t>0,\  \sup\limits_{y\in\R^n}u(t,y)= 0.
\end{equation}
We deduce
\begin{equation*}
\int_0^t\rho=\sup_{y\in\R^n}U(t,y).
\end{equation*}
Therefore, $\int_0^t\rho$ does not depend on the extracted subsequence, and the convergence occurs for the whole sequence $\eps\to0$, which achieves the proof.
\end{proof}

Notice that the bound on $\rho(s)$ can be made more precise in the Cesaro sense
\begin{equation}
\forall t>0,\quad -\overline{\Lambda}\leq\frac{1}{t}\int_0^t \rho(s)ds\leq-\underline{\Lambda}
\end{equation}
since, from~\autoref{TheoremUeps} we know that  $\int_0^t \rho=\sup_{\R^n} U(t,\cdot)\leq \sup_{\R^n} U^0(\cdot)-\underline{\Lambda}t=-\underline{\Lambda}t$ and the bound from below is  similar.
%

\subsection{Adaptive dynamics}\label{SecAdaptativeDynamics}
We now prove the third statement of~\autoref{MainResults}. We need further assumptions on the initial conditions, 
\begin{equation}\label{initialu}
\exists! \ \bar y^0_\eps\in\R^n,\quad U^0_\eps(\bar y^0_\eps)=\max\limits_{y\in\R^n} U_\eps^0(y)=0,
\end{equation}
\begin{equation}\label{initial_convergence_y}
\bar y^0_\eps\text{ converges to some }\bar y^0\in\R^n\text{ when $\eps$ vanishes},
\end{equation}
\begin{equation}\label{initial_concavity}
\nabla_y^2 U^0(\bar y^0)<0.
\end{equation}

\begin{proposition}\label{TheoremAD}
Under the assumptions of section~\ref{SecAssumptions} and~\eqref{initialu}--\eqref{initial_concavity}, for a short time interval $[0,T]$, there exists a unique $\bar y(t)\in\R^n$ on which $U(t,\cdot)$ reaches its maximum. Moreover, $t\mapsto\bar y(t)\in\mathcal{C}^1$ and satisfies the Canonical Equation
\begin{equation}\label{CanonicalEquation}
\left\{\begin{aligned}
&\frac{\mathrm{d}}{\mathrm{d}t}\bar y(t)=\left(\nabla_y^2U(t,\bar y(t))\right)^{-1}\cdot\nabla_y\Lambda(\bar y(t),1)+\D_\eta\Lambda(\bar y(t),1)\int_{\R^n}M(z)z dz,\\
&\bar y(0)=\bar y^0.
\end{aligned}\right.
\end{equation}.
\end{proposition}
Note that~\eqref{CanonicalEquation} features a drift term $\D_\eta\Lambda\int_{\R^n}M(z)zdz$. If the mutation kernel $M(\cdot)$ is even, this term vanishes and we recover the classical Canonical Equation. Let us also recall that uniqueness and regularity of a unique concentration point (monomorphism) is a hard questions with few progresses, see \cite{GB.BP:07,MiRo2016, CaLa}.

\begin{proof}
Since $U$ (defined in \autoref{TheoremU}) satisfies \eqref{HJU} with smooth initial datum, it is uniformly $C^2$ in the $y$ variable for short times $t\in[0,T]$, $T>0$ (this can be proved with the method of the characteristics, see Chapter 3.2 in \cite{LE:98}).

Consider such a time interval $[0,T]$, and $V\subset\R^n$ a neighborhood of $y^0$.
We are interested in the solutions $(t,\bar y)\in [0,T]\times V$ of
\begin{equation}\label{EquationBut}
\nabla_y U(t,\bar y)=0.
\end{equation}
From \eqref{initialu} we know that at initial time there exists a unique solution $\bar y^0$ of \eqref{EquationBut}. Besides, $\nabla_y^2 U^0(\bar y^0)$ is invertible. From the implicit functions theorem, there exists a unique $\bar y(t)\in\R^n$ satisfying \eqref{EquationBut}, for $t$ in a certain time interval still denoted $[0,T]$. We can again choose a smaller $T$ to ensure that $\bar y(t)$ remains in $V$. 

From \eqref{initial_concavity}, we can also choose $T$ and $V$ small enough to guarantee, $\forall t\in[0,T]$,
\begin{gather*}
\bar y(t)\in V,\quad U(t,\cdot) \text{ is striclty concave in }V,\\
\max_{y\in V} U(t,y)=\max_{y\in\R^n} U(t,y).
\end{gather*}
Hence, for all $t\in[0,T]$, the solution $\bar y(t)$ of \eqref{EquationBut} must satisfy
\begin{equation}\label{PropMaximum}
U(t,\bar y(t))=\max_{y\in V} U(t,y)=\max_{y\in\R^n} U(t,y)=\int_0^t \rho,
\end{equation}
which proves the first part of the proposition and that $t\mapsto \bar y (t)$ is $\mathcal{C}^1$. 

For the  Canonical Equation, we differentiate \eqref{EquationBut} with respect to $t$, and noting that
\begin{equation}
\eta(t,\bar y(t))=\int_{\R^n}M(z)e^{\nabla_yU(t,\bar y(t))\cdot z}dz =1,
\end{equation}
we obtain
 \begin{align*}
     0
     &=\frac{\mathrm{d}}{\mathrm{d}t}\left[\nabla_y U(t,\bar y (t))\right]\\
     &=-\nabla_y\Lambda(\bar y (t),1)-\nabla_y^2 U(t,\bar y (t))\cdot\D_\eta\Lambda(\bar y(t),1)\int_{\R^n}M(z)zdz+\nabla_y^2 U(t,\bar y (t))\cdot \dot{\bar y}(t),
 \end{align*}
and \eqref{CanonicalEquation} follows, for $t\in[0,T]$. 
\end{proof}

\begin{corollary}
Under the same assumptions, $t\mapsto \rho(t)\in C^1([0,T])$ and for all $t\in[0,T]$,
\begin{equation}
\rho(t)=-\Lambda(\bar y(t),1)
\end{equation}
and
\begin{equation}
\frac{\mathrm{d}}{\mathrm{d}t}\rho(t)=-\nabla_y\Lambda\cdot \nabla_y^2 U\cdot \nabla_y\Lambda -\D_\eta\Lambda \left(\int_{\R^n} M(z) z dz\right)\cdot \left(\nabla_y^2 U\right)^2\cdot \nabla_y \Lambda.
\end{equation}
where $\nabla_y^2 U$ is evaluated in $(t,\bar y(t))$, and the derivatives of $\Lambda$ in $(\bar y(t),1)$.

In particular, if $M(\cdot)$ is even, then $\frac{\mathrm{d}}{\mathrm{d}t}\rho(t)\geq 0$ and $-\overline{\Lambda}\leq \rho\leq -\underline{\Lambda}$.
\end{corollary}
\begin{proof}
Follows from $\D_t U(t,\bar y (t))=\rho(t)$ and $\frac{\mathrm{d}}{\mathrm{d}t}\left[\D_t u(t,\bar y (t))\right]=\frac{\mathrm{d}}{\mathrm{d}t}\rho(t).$
\end{proof}

\begin{remark}
The limitation of \autoref{TheoremAD} to a short time interval is merely due to three independant phenomena. First, the possible loss of concavity, or apparition of singularities for $U$, coming from the Hamilton-Jacobi equation~\eqref{HJU}. Secondly, the possible "jump" of the point where $U$ reaches its maximum, contradicting $\max_{y\in V} U(t,y)=\max_{y\in\R^n} U(t,y)$ in \eqref{PropMaximum}. Finally, the possible blow-up in finite time of $\bar y(t)$ from the dynamics of the Canonical Equation~\eqref{CanonicalEquation}. Regarding the last point, we point out that $\Lambda$ can sometimes be used as a Lyapunov function. Indeed, 
we have
\begin{align*}
&\frac{\de}{\de t}\left[\Lambda(\bar y(t),1)\right]=\nabla_y\Lambda(\bar y(t),1)\cdot\dot{\bar y}(t)\\
&=\nabla_y\Lambda(\bar y(t),1)\cdot\nabla_y^2U^{-1}\cdot\nabla_y\Lambda(\bar y(t),1)+\D_\eta\Lambda(\bar y(t),1)\nabla_y\Lambda(\bar y(t),1)\cdot\int_{\R^n}M(z)z dz\\
&\leq \D_\eta\Lambda(\bar y(t),1)\nabla_y\Lambda(\bar y(t),1)\cdot\int_{\R^n}M(z)z dz.
\end{align*}
In particular, if $M(\cdot)$ is even, then $\frac{\de}{\de t}\left[\Lambda(\bar y(t),1)\right]\leq0$. Thus, if $\bar y^0$ belongs to a "well" of $\Lambda$, then $\bar y(t)$ remains "trapped", which prevents from an blow-up in finite time. It also implies, at least formally, that $\bar y(t)$ converges to a local minimum of $\Lambda(\cdot,1)$ when $t\to+\infty$.

\end{remark}

\section{Construction, estimates, and asymptotics of $U_\eps$ - Proof of \autoref{th:MainResultsConvergenceU}}\label{sec:appendixConcentrationWithMutations}

\subsection[The eigenproblem]{The eigenproblem - proof of \autoref{ThEigenElements}}\label{sec:AppendixEigenElements}

An immediate calculation on \eqref{EigenProblemFinal} gives the explicit solution~\eqref{DefinitionQ} for $Q$ in terms of $\Lambda$.
Multiplying by $b(x,y)$ and integrating in $x$, we obtain formula~\eqref{FormuleImplicite}. 
Next, from \eqref{AssumptionFWith} we have, for $y\in\R^n,\ \forall\eta\in (\underline\eta(y),\overline\eta(y))$
\begin{equation}
F(y,\underline\Lambda)\leq\frac{1}{\eta}\leq F(y,\overline\Lambda).
\end{equation} 
As $\D_\lambda F>0$, we conclude the existence and uniqueness of  $\Lambda(y,\eta)$ as the unique solution of~\eqref{FormuleImplicite}. Now, using \eqref{DefinitionQ}, we obtain  existence and uniqueness of $Q$.

Finally, the bounds \eqref{Prerequis:BoundsOnLambda} follow from $\D_\lambda F>0$ and
\begin{equation}
F(y,\underline\Lambda)\leq F(y,\Lambda(y,\eta))\leq F(y,\overline\Lambda).
\end{equation}

For later purpose, we also neeed the following stronger version of~\eqref{IntermediateInequality}.
\begin{lemma}\label{IntermediateInequalitySTRONG}
There exists $\delta>0$ such that, for all $y\in\R^n$, $\eta\in (\underline{\eta}(y),\overline{\eta}(y))$, we have
\begin{equation}
\D_\eta^2\Lambda(y,\eta)+\frac{\D_\eta\Lambda(y,\eta)}{\eta}\leq -\delta.
\end{equation}
\end{lemma}

\begin{proof}
From the proof of~\eqref{IntermediateInequality}, we have
\begin{equation}
\D_\eta^2\Lambda+\frac{\D_\eta\Lambda}{\eta}= \frac{1}{(\D_\lambda F)^3\eta^3}\left((\D_\lambda F)^2-F\D_\lambda^2F\right).
\end{equation}
Since $\frac{1}{(\D_\lambda F)^3\eta^3}\geq \frac{1}{L^3\overline{\eta}^3}$, our goal is to show that
\begin{equation}
(\D_\lambda F)^2-F\D_\lambda^2F\leq -\delta,
\end{equation}
for all $y\in\R^n$, $\eta\in (\underline{\eta}(y),\overline{\eta}(y))$.

We set, for all $x\geq0$, $y\in\R^n$, $\lambda\in\R$, 
\begin{equation}
f(x,y,\lambda):=\frac{b(x,y)}{A(x,y)}\exp\left(\int_0^x\frac{\lambda-d(x',y)}{A(x',y)}dx'\right).
\end{equation}
According to~\eqref{DefinitionF}, we have $F(y,\lambda)=\int_{\R_+}f(x,y,\lambda)dx$. We also define the probability measure $\tilde f(x,y,\lambda):=\frac{f(x,y,\lambda)}{F(y,\lambda)}.$ Now, setting $\mathcal{A}(x,y):=\int_0^x\frac{1}{A(x',y)}dx'$, we have
\begin{equation}
\D_\lambda F(y,\lambda)= \int_{\R_+}\mathcal{A}(x,y)f(x,y,\lambda)dx,\quad \D^2_\lambda F(y,\lambda)= \int_{\R_+}\mathcal{A}(x,y)^2f(x,y,\lambda)dx.
\end{equation}
It gives,
\begin{equation}
(\D_\lambda F)^2-F\D_\lambda^2F=-F^2\int_{\R_+}\left(\mathcal{A}(x,y)-\int_{\R_+}\mathcal{A}\tilde f dx\right)^2\tilde{f}dx
\end{equation}
(the function are evaluated on $y$ and $\lambda=\Lambda(y,\eta)$).
We have $F^2\leq \frac{1}{\underline{\eta}^2}$, and assumption~\eqref{AssumptionANoDegenerate} implies that the above term is negative uniformly in $y\in\R^n$, $\eta\in (\underline{\eta}(y),\overline{\eta}(y))$.
\end{proof}

\subsection[Construction of $U_\eps$]{Construction of $U_\eps$ - proof of \autoref{TheoremUeps}}\label{sec:ConstructionU}

We present a proof of existence based on a regularization argument. For the constrained Hamilton-Jacobi equation, a fixed point method has also been proposed in~\cite{KimACAP2019}. Our proof is divided into three parts. First, we construct $U_\eps$ on a truncated problem. Then, we prove a uniform a priori estimate on $\D_t U_\eps$, which allows finally to remove the truncation.

\paragraph{Extending $\Lambda$.}
\autoref{ThEigenElements} defines $\Lambda(y,\eta)$ only for $y\in\R^n$ and $\eta\in\left(\underline{\eta}(y),\overline{\eta}(y)\right)$. We first need to artificially extend $\Lambda(y,\eta)$ for $\eta\in(0,+\infty)$. For $y\in\R^n$, we set
\begin{equation*}
\tilde\Lambda(y,\eta)=
\begin{cases}
\underline{\Lambda}-\underline{B_y}(\eta) &if $\eta<\underline\eta(y)$,\\
\Lambda(y,\eta)&if $\underline\eta(y)\leq\eta\leq\overline\eta(y)$,\\
\overline{\Lambda}+\overline{B_y}(\eta) &if $\eta>\overline\eta(y)$,
\end{cases}
\end{equation*}
where $\underline{B_y}$ and $\overline{B_y}$ are chosen to be positive, increasing, bounded by $1$, and such that $\tilde\Lambda$ is smooth. Note that the extension of $\Lambda$ is completely arbitrary, but we will show, a posteriori, that $\eta_\eps\in\left(\underline{\eta}(y),\overline{\eta}(y)\right)$.
We consider the following problem
\begin{equation}\label{TruncatedProblem}
\left\{
\begin{aligned}
&\D_t \tilde U_\eps(t,y)=-\tilde\Lambda\left(y,\int_{\R^n}M(z)e^{\frac{\tilde U_\eps(t,y+\eps z)-\tilde U_\eps(t,y)}{\eps}}dz\right), &\forall t\geq0,\ \forall y\in\R^n,\\
&\tilde U_\eps(0,y)=u_\eps^0(y), &\forall y\in\R^n.
\end{aligned}\right.
\end{equation}

\paragraph{Solution for the truncated problem.}

For a fixed $R>0$, we consider a truncation function $\phi_R:\R\to\R$ which is smooth, increasing and satisfies the following conditions:
\begin{itemize}
\item $\phi_R(r)= r \text{ for }r \in[-\frac{R}{2},\frac{R}{2}]$,
\item $\phi_R(r)=R \text{ for }r \geq 2R$,
\item $\phi_R(r)= -R \text{ for } r \leq -2R$,
\item $\phi_R'\geq 0$ is uniformly bounded.
\end{itemize}    
For $\eps>0$, we consider the Cauchy problem
\begin{equation}\label{eqRter}
    \left\{
    \begin{aligned}
        &\D_t \tilde U^R_\eps(t,y)=\phi_R\left(-\tilde\Lambda\left(y,\int_{\R^n}{M(z)e^{\frac{\tilde U^R_\eps(t,y+\eps z)-\tilde U^R_\eps(t,y)}{\eps}}d z}\right)\right),\\
        &\tilde U^R_\eps(0,\cdot)=U^0_\eps.      
    \end{aligned}\right.
\end{equation}
for which the classical Cauchy-Lipschitz theorem provides existence and uniqueness of a solution $\tilde U^R_\eps$, defined globally in time.

\paragraph{Estimate on the time derivative.}

\begin{lemma}\label{BoundOnD_tUBIS}
We have, with $
\D_t U^0_\eps:=-\Lambda\left(y,\eta_\eps^0(y)\right)
$
and $\eta_\eps^0$ is defined in \eqref{DefinitionEtaLambda0}, 
\begin{equation}
\inf\limits_{y\in\R^n} \D_t U^0_\eps(y)\leq\D_t \tilde U^R_\eps(t,y)\leq \sup\limits_{y\in\R^n} \D_t U^0_\eps(y),\quad \forall \eps>0,\ t>0,\ y\in\R^n .
\end{equation}

\end{lemma}
The full proof of this statement, which is technical, can be found in~\cite{Nordmann2018a} (proof of Proposition~4.7, Appendix~D). We give the formal idea of the method.

Let us fixe $\eps>0$, $R>0$ and set $V(t,y):=\D_t \tilde U^R_\eps(t,y).$
Differentiating \eqref{eqRter} with respect to $t$, we obtain
\begin{equation}\label{formal_timebis}
\D_tV(t,y)=\int_{\R^n} K (t,y,z) \left(\frac{V(t,y+\eps z)-V(t,y)}{\eps}\right) dz,
\end{equation}
where $K(t,y,z):=-\phi'_R\ \D_\eta\tilde\Lambda\ M(z)e^{\frac{\tilde U^R_\eps(t,y+\eps z)-\tilde U^R_\eps(t,y)}{\eps}} $. Since $\D_\eta\Lambda<0$, we have $K\geq0$.
Then, if for some $t>0$, $V(t, \cdot)$ reaches its maximum at $\bar y\in\R^n$, we obtain the inequality 
\begin{equation*}
\D_t V(t, \bar y) =\int_{\R^n} K (t,\bar y,z) \left(\frac{V(t,\bar y+\eps z)-V(t,\bar y)}{\eps}\right) dz \leq 0.
\end{equation*}
Formally, it shows that the maximum value of $V$ is decreasing with time, that is, 
\[
\sup_y V(t, y) \leq \sup_y V(0, y)= \sup_{y}\D_tU^{0}_\eps.
\]
With the same method we show
${\inf_yV\geq\inf_y\D_tU^{0}_\eps}$, which conclude the proof of \autoref{BoundOnD_tUBIS}.
\\

Hereafter, from assumption~\eqref{AssumptionFWith} and $\D_\lambda F>0$, we have
\begin{equation}\label{BoundD_tU^0}
-\overline{\Lambda}\leq \D_tU^0_\eps(y)\leq -\underline{\Lambda}.
\end{equation}
Using \autoref{BoundOnD_tUBIS}, we infer
\begin{equation}\label{EstimateIntermediaire}
-\overline{\Lambda}\leq\phi_R\left(-\tilde \Lambda\left(y,\int_{\R^n}M(z)e^{\frac{\tilde U^R_\eps(t,y+\eps z)-\tilde U^R_\eps(t,y)}{\epsilon}}\right)\right)\leq-\underline{\Lambda}.
\end{equation}

\paragraph{Removing the truncation.}

From~\eqref{EstimateIntermediaire} and the choice of $\phi_R$, for $R$ large enough, we have
\begin{equation}
-\phi_R\left(\tilde \Lambda\left(y,\int_{\R^n}M(z)e^{\frac{\tilde U^R_\eps(t,y+\eps z)-\tilde U^R_\eps(t,y)}{\epsilon}}\right)\right)=\tilde \Lambda\left(y,\int_{\R^n}M(z)e^{\frac{\tilde U^R_\eps(t,y+\eps z)-\tilde U^R_\eps(t,y)}{\epsilon}}\right).
\end{equation}
Besides, since $\D_\eta\tilde\Lambda<0$, we have
\begin{equation}\label{EstimateIntermediaire2}
\underline{\eta}(y)\leq \int_{\R^n}M(z)e^{\frac{\tilde U^R_\eps(t,y+\eps z)-\tilde U^R_\eps(t,y)}{\epsilon}}\leq \overline{\eta}(y),
\end{equation}
for all $R$ large enough, $\eps>0$, $t\geq0,\ y\in\R^n$. Thus, from the definition of $\tilde{\Lambda}$ in \eqref{TruncatedProblem}, we have
\begin{equation*}
-\tilde\Lambda\left(y,\int_{\R^n}M(z)e^{\frac{\tilde U^R_\eps(t,y+\eps z)-\tilde U^R_\eps(t,y)}{\eps}}dz\right)=-\Lambda\left(y,\int_{\R^n}M(z)e^{\frac{\tilde U^R_\eps(t,y+\eps z)-\tilde U^R_\eps(t,y)}{\eps}}dz\right),
\end{equation*}
that is, $\tilde U_\eps^R$ is a solution of \eqref{DefU_eps}. The proof is thereby achieved.

\subsection[A priori Lipschitz estimate]{A priori Lipschitz estimate - proof of \autoref{UepsLipschitz}}\label{sec:LipschitzEstimateEps}
We follow the same idea as for \autoref{BoundOnD_tUBIS}. However, there are some technical difficulties. First, we have to deal with a "source term" $\nabla_y\Lambda\left(y,\eta_\eps(t,y)\right)$ which is bounded by a constant $\frac{L}{l\underline{\eta}^2}$, using~\eqref{AssumptionDerF}, \eqref{AssumptionDerFy2}, \eqref{FormuleDeriveeLambda} and \eqref{BoundOnEta}. In addition, we first need to prove the estimate on a truncated function, then to remove the truncation. 
%
%

We fix $i\in\{1,\dots,n\}$, $T>0$, and we set 
\begin{equation*}
\begin{aligned}
W_\eps(t,y):= \D_{y_i}{U}_\eps (t,y) .
\end{aligned}
\end{equation*}
Differentiating \eqref{DefU_eps}, we obtain
\begin{equation}\label{eqRDbis}
\begin{aligned}
&\D_t{W}_\eps(t,y)\\
&=-\D_{y_i}\Lambda(y,\eta_\eps)-\D_\eta\Lambda(y,\eta_\eps)\left(\int M(z)e^{\frac{{U}_\eps(t,y+\eps z)-{U}_\eps(t,y)}{\eps}}\left[\frac{{W}_\eps(t,y+\eps z)-{W}_\eps(t,y)}{\eps}\right]dz\right)\\ 
&:=\mathcal{F}(t,y,W_\eps(t,\cdot)).
\end{aligned}
\end{equation}
We formally define a truncated problem, for $R>0$, and its solution $W^R_\eps$ satisfying
\begin{equation}\label{troncatureS}
\begin{aligned}
&W^R_\eps(t,y)=\phi_R\left(\D_{y_i}U^0_\eps(y)+\int_0^t \mathcal{F}(s,y,{W}_\eps^R(s,\cdot))ds\right),\\
\end{aligned}
\end{equation}
where $\mathcal{F}$ is defined above and $\phi_R$ is a truncation function as in~\eqref{eqRter}. We can prove existence and uniqueness of a global solution of~\eqref{troncatureS} by a direct application of the Cauchy-Lipschitz theorem. 
\\

We set $\bar W^R_\eps:=W^R_\eps-Ct$ with $C:=\frac{L}{l\underline{\eta}^2}$. Our goal is to show
\begin{equation}\label{objectifbis}
\begin{aligned}
& \forall t\in[0,T],\forall y\in\R^n, & \bar W^R_\eps(t,y)\leq\sup  \D_{y_i}U_\eps^0.
\end{aligned}
\end{equation}
By contradiction, assume  \eqref{objectifbis} does not hold, i.e., there exist ${y_0\in\R^n},{t_0\in[0,T]}$ such that
\begin{equation}\label{absurd2'}
\bar W^R_\eps(t_0,y_0)-\sup \D_{y_i}U_\eps^{0}>0.
\end{equation}
For $\beta>0$, $\alpha>0$ small enough, $t \in [0,T], y\in\R^n$ we introduce
\begin{equation*}
\varphi_{\alpha,\beta}(t,y):=\bar W^R_\eps(t,y)- \alpha t-\beta\vert y-y_0\vert.
\end{equation*}
As $\bar W^R_\eps$ is bounded, $\varphi_{\alpha,\beta}$ reaches its maximum on $[0,T]\times\R^n$ at a point $(\bar t, \bar y).$ 
We have 
\begin{equation*}
\forall  z\in\R^n,\    \varphi_{\alpha,\beta}(\bar t,\bar y+ \eps z)\leq\varphi_{\alpha,\beta}(\bar t,\bar y).
\end{equation*}
Then, we obtain the inequality
\begin{equation*}\label{accroissement}
\forall  z\in\R^n,\ \frac{\bar W^R_\eps( \bar t, \bar y+\eps z)-\bar W^R_\eps( \bar t, \bar y)}{\eps}\leq\beta\frac{\vert \bar y +z-y_0\vert-\vert \bar y -y_0\vert}{\eps}\leq\beta \vert z\vert.
\end{equation*}
        
        We choose $\alpha$ small enough so that $\varphi_{\alpha,\beta}(t_0,y_0)>\varphi_{\alpha,\beta}(0,y_0)= \D_{y_i}U_\eps^{0}(y_0),$ which is possible thanks to \eqref{absurd2'}. It implies $\bar t>0$. Hence ${\D_t\varphi_{\alpha,\beta}(\bar t,\bar y)\geq0}$, i.e. ${\D_t\bar W^R( \bar t, \bar y)\geq\alpha}$ (if $\bar t=T$, then $\D_t$ stands for the left derivative).
Differentiating \eqref{troncatureS} at $(\bar t,\bar y)$, we have
\begin{equation*}
\begin{aligned}
\alpha&\leq\D_t\bar W^R_\eps(\bar t,\bar y)\\
&\leq -\sup\phi'_R\times\D_{y_i}\Lambda(y,\eta_\eps(t,y))- C\\
&\quad +\sup\phi'_R\times\left(-\D_\eta\Lambda(y,\eta)\right)\int M(z)e^{\frac{ U_\eps(\bar t,\bar y+\eps z)- U(\bar t,\bar y)}{\eps}}\left[\frac{W^R_\eps(\bar t,\bar y+\eps z)-W^R_\eps(\bar t,\bar y)}{\eps}\right]dz.
        \end{aligned}
    \end{equation*}
Now, from $\vert\D_{y_i}\Lambda(y,\eta_\eps(t,y))\vert\leq\frac{L}{l\underline{\eta}^2}= C$ and $0\leq-\D_\eta\Lambda(y,\eta)\leq\frac{L}{\underline{\eta}^2}$, we have
\begin{equation*}
\begin{aligned}        
            \alpha&\leq \frac{L}{\underline{\eta}^2}\left(\int M(z)e^{\frac{ U_\eps(\bar t,\bar y+\eps z)- U(\bar t,\bar y)}{\eps}}\vert z\vert\de z\right)\times \beta\\
            &\leq \frac{L}{\underline{\eta}^2} \left(\int{M(z)e^{\frac{-2\underline\Lambda T}{\eps}+k_0\vert z\vert}\vert z\vert\de z}\right)\times \beta.
        \end{aligned}
    \end{equation*}
Then, passing to the limit $\beta\to 0$ we obtain ${\alpha\leq0}$: contradiction. Thus, we have 
\begin{equation*}
\bar W_\eps^{R}\leq\sup\vert \D_{y_i}U_\eps^{0}\vert=k^0.
    \end{equation*}
    We proceed similarily to obtain the reverse inequality $\bar W_\eps^{R}\geq k^0$. We have, for all $R>0,\ \eps>0,\ t\in[0,T],\ y\in\R^n$
    \begin{equation*}
        \vert W^R_\eps(t,y)\vert \leq k^0+Ct.
    \end{equation*}

Finally, the bound on $W^R_\eps$ is uniform in $R$ so we can remove the truncation, as detailed in section~\ref{sec:ConstructionU}. Thus, $W^R_\eps=W_\eps$ for $R$ large enough and
    \begin{equation*}
        \vert \D_{y_i} U_\eps(t,y)\vert \leq k^0+Ct.
    \end{equation*}

\subsection[Semi-convexity]{Semi-convexity - proof of \autoref{ThmSemiConvexity}}\label{sec:W11estimate}

For convex Hamiltonian, the semi-convexity of the solution is a classical matter, \cite{GB:94, Cannarsa2004}. Here, we have to deal with a nonlocal operator which features a difference rather than a gradient

\paragraph{Semi-convexity in $t$.}
For shorter formulas, we need some notations
\begin{equation}
V_\eps:= \D_t^2U_\eps, \qquad J_\eps(t,y,z):=M(z)e^{\frac{U_\eps(t,y+\eps z)-U_\eps(t,y)}{\eps}}.
\end{equation}
We begin with two results that are used later and express some properties usually connected to the  convexity of the Hamiltonian. Firstly we observe that 
\begin{lemma}\label{LemmaIntermediaire2}
For all $\eps>0$, $t>0$, $y\in\R^n$, we have
\begin{equation}
(\D_t\eta_\eps)^2\leq \eta_\eps\int_{\R^n}\left(\frac{\D_tU_\eps(t,y+\eps z)-\D_t U_\eps(t,y)}{\eps}\right)^2J_\eps dz.
\end{equation}
\end{lemma}
\begin{proof}
Use Jensen's inequality and the definition of $\eta_\eps$ in~\eqref{DefinitionEtaLambda}.
\end{proof}
Next, we prove that 
\begin{equation}\label{InequationV}
\D_t V_\eps \geq -\D_\eta\Lambda(y,\eta_\eps) \int_{\R^n}\frac{V_\eps(t,y+\eps z)-V_\eps(t,y)}{\eps}J_\eps dz.
\end{equation}
\begin{proof} Differentiating~\eqref{DefU_eps} twice, we find
\begin{equation}\label{EquationOnD_T2U}
\begin{aligned}
\D_t V_\eps= 
&-\D_\eta\Lambda \int_{\R^n}\frac{V_\eps(t,y+\eps z)-V_\eps(t,y)}{\eps}J_\eps dz\\
&-\D^2_{\eta}\Lambda\left(\D_t\eta_\eps\right)^2
-\D_\eta\Lambda\int_{\R^n}\left(\frac{\D_tU_\eps(t,y+\eps z)-\D_t U_\eps(t,y)}{\eps}\right)^2J_\eps dz.
\end{aligned}
\end{equation}
Next, combining \autoref{LemmaIntermediaire2} with \eqref{IntermediateInequality}, we find
\begin{equation}\label{LemmaHamiltonienConvex}
\D^2_{\eta}\Lambda(y,\eta_\eps)\left(\D_t\eta_\eps\right)^2
+\D_\eta\Lambda(y,\eta_\eps)\int_{\R^n}\left(\frac{\D_tU_\eps(t,y+\eps z)-\D_t U_\eps(t,y)}{\eps}\right)^2J_\eps dz
\leq0,
\end{equation}
Using the above inequality and~\eqref{EquationOnD_T2U}, we find~\eqref{InequationV}.
\end{proof}

From inequality~\eqref{InequationV}, we deduce
\begin{lemma}\label{lemmaSemiConvex}
$V_\eps$ is  uniformly bounded from below and more precisely, with $V_\eps^0(y):=V_\eps(t=0,y)$, we have
\begin{equation}\label{objectif}
V_\eps(t,y)\geq\inf_{y\in \R^n}V^0_\eps(y)> -\infty,\quad \forall \eps>0,\ \forall t>0,\ \forall y \in \R^n.
\end{equation}
\end{lemma}

The proof follows closely the method of section~\ref{sec:LipschitzEstimateEps}. The formal idea is the following. If, for some $t>0$, $V_\eps (t, \cdot)$ reaches its minimum at $\bar y\in\R^n$, from~\eqref{InequationV} we obtain $\D_t V_\eps(t, \bar y)\geq0$.
Formally, it shows that the minimum value of $V_\eps$ is increasing with time, that is, $\inf_y V_\eps(t, y) \geq \inf_y V_\eps(0, y)$. Then, we conclude with the fact that $\inf_y V_\eps(0, y)$ is bounded, uniformly in $\eps>0$.

\begin{proof}
Differentiating~\eqref{LinkVEta} in $t$, we obtain
    \begin{equation}\label{LinkVEtaExplicite}
        V_\eps(t,y)=-\D_\eta\Lambda(y,\eta_\eps)\int_{\R^n}\frac{\Lambda_\eps(t,y+\eps z)-\Lambda_\eps(t,y)}{\eps} J_\eps dz.
    \end{equation}
In particular, our assumptions imply $\inf_{y\in \R^n}V^0_\eps>-\infty$ uniformly in $\eps>0$, thus \eqref{objectif} implies that $V_\eps$ is bounded from below, uniformly in $\eps$.
\\
    
We prove \eqref{objectif} by contradiction. We assume that there exists $(T,y_0)\in(0,+\infty)\times\R^n$ such that
    \begin{equation}\label{absurd2}
        V_\eps(T,y_0)-\inf_{y\in \R^n} V_\eps^0(y)<0.
    \end{equation}
    For $\beta>0$, $\alpha>0$ small and for $t\in[0,T], y\in\R^n,$ we also introduce
    \begin{equation*}
            \varphi_{\alpha,\beta}(t,y):=V_\eps(t,y)+ \alpha t+\beta\vert y-y_0\vert.
    \end{equation*}
    
From~\eqref{LinkVEtaExplicite} and for a fixed $\eps>0$, we have $V_\eps(t,y)$ is bounded from below uniformly in $t\in[0,T]$, $y\in\R^n$. Therefore, $\varphi_{\alpha,\beta}$ goes to $+\infty$ as $\vert y\vert \to +\infty$ and reaches its minimum on ${[0,T]\times\R^n}$ at a point $(\bar t, \bar y)$. We have
\begin{equation*}
    \varphi_{\alpha,\beta}(\bar t,\bar y+\eps z)\geq\varphi_{\alpha,\beta}(\bar t,\bar y),\quad\forall z\in\R^n,
\end{equation*}
thus
\begin{equation}\label{accroissement1}
    \frac{V_\eps(\bar t,\bar y+\eps z)-V_\eps(\bar t,\bar y)}{\eps}\geq\beta\frac{\vert\bar y-y_0\vert-\vert\bar y-y_0+\eps z\vert}{\eps}\geq-\beta \vert z\vert,\quad \forall z\in\R^n.
\end{equation}

We choose $\alpha$ small enough to ensure $\varphi_{\alpha,\beta}(T,y_0)<\varphi_{\alpha,\beta}(0,y_0),$ which is possible thanks to assumption \eqref{absurd2}. It implies $\bar t>0$. Hence ${\D_t\varphi_{\alpha,\beta}(\bar t,\bar y)\leq0}$, that is ${\D_tV_\eps(\bar t,\bar y)\leq-\alpha}$ (if $\bar t=T$ then $\D_t V_\eps^R(\bar t,\bar y)$ stands for the left-derivative).
From~\eqref{InequationV} at $(\bar t,\bar y)$, using \eqref{accroissement1},
\begin{equation}
    \begin{aligned}
        -\alpha&\geq\D_tV_\eps(\bar t,\bar y)=-\D_\eta\Lambda \int_{\R^n}M(z)e^{\frac{U_\eps(t,y+\eps z)-U_\eps(t,y)}{\eps}}\frac{V_\eps(t,y+\eps z)-V_\eps(t,y)}{\eps}dz\\
        &\geq\beta \inf\left[-\D_\eta\Lambda\right]\int{M(z)e^{\frac{U_\eps(\bar t,\bar y+\eps z)-U_\eps(\bar t,\bar y)}{\eps}}\vert z\vert d z}\\
        &\geq \beta \frac{1}{L\overline{\eta}}\int{M(z)e^{\left(k_0+\frac{L}{l\underline{\eta}^2}T\right)\vert z\vert}\vert z\vert d z}
    \end{aligned}
\end{equation}
where, in the last step, we used \autoref{UepsLipschitz} and $-\D_\eta\Lambda\geq \frac{1}{L\overline{\eta}}$.
As $\beta$ goes to $0$, we obtain $\alpha\leq0$, which is absurd. The proof is thereby achieved.
\end{proof}

\paragraph{Semi-convexity in $y$.}

Let us show how the method can be adapted to prove that $\D^2_i U_\eps$ (the second derivative w.r.t. $y_i$) is bounded from below.
We set
$
W_\eps:= \D_i^2U_\eps.
$
Differentiating~\eqref{DefU_eps} twice, we find
\begin{equation}\label{EquationOnD_y2U}
\begin{aligned}
\D_t W_\eps= 
&-\D_\eta\Lambda \int_{\R^n}\frac{W_\eps(t,y+\eps z)-W_\eps(t,y)}{\eps}J_\eps dz\\
&-\D^2_{\eta}\Lambda\left(\D_i\eta_\eps\right)^2
-\D_\eta\Lambda\int_{\R^n}\left(\frac{\D_iU_\eps(t,y+\eps z)-\D_i U_\eps(t,y)}{\eps}\right)^2J_\eps dz\\
&-\D^2_i\Lambda - 2\D^2_{i,\eta}\Lambda \D_i\eta_\eps.
\end{aligned}
\end{equation}
In contrast with the equation~\eqref{EquationOnD_T2U} on $\D^2_t U_\eps$, we need to deal with a source term and a linear term in the last line.

For any constant $K>0$, Young's inequality implies
\begin{equation}\label{YoungInequality}
- 2\D^2_{i,\eta}\Lambda \D_i\eta_\eps\geq -K^2\vert \D^2_{i,\eta}\Lambda\vert ^2-\frac{1}{K^2}(\D_i\eta_\eps)^2.
\end{equation}
Applying this inequality and~\autoref{LemmaIntermediaire2} (replacing $\D_t$ by $\D_i$) in~\eqref{EquationOnD_y2U}, we obtain
\begin{equation}\label{EquationOnD_y2UBIS}
\begin{aligned}
\D_t W_\eps\geq  
&-\D_\eta\Lambda \int_{\R^n}\frac{W_\eps(t,y+\eps z)-W_\eps(t,y)}{\eps}J_\eps dz\\
&-\left(\D^2_{\eta}\Lambda+\frac{\D_\eta\Lambda}{\eta_\eps}+\frac{1}{K^2}\right)\left(\D_i\eta_\eps\right)^2\\
&-\D^2_i\Lambda -K^2\vert \D^2_{i,\eta}\Lambda\vert ^2.
\end{aligned}
\end{equation}

Using lemma \autoref{IntermediateInequalitySTRONG}, and choosing $K\geq\frac{1}{\sqrt\delta}$, we have 
\begin{align*}
\D_t W_\eps\geq  
&-\D_\eta\Lambda \int_{\R^n}\frac{W_\eps(t,y+\eps z)-W_\eps(t,y)}{\eps}J_\eps dz\\
&-\D^2_i\Lambda -K^2\vert \D^2_{i,\eta}\Lambda\vert ^2.
\end{align*}
From assumptions~\eqref{AssumptionDerFy2}-\eqref{AssumptionSemiConvexity}, the source term $-\D^2_i\Lambda -K^2\vert \D^2_{i,\eta}\Lambda\vert ^2$ is bounded from below by some constant $-K'<0$.
We end up with
\begin{equation}
\D_t W_\eps\geq  
-\D_\eta\Lambda \int_{\R^n}\frac{W_\eps(t,y+\eps z)-W_\eps(t,y)}{\eps}J_\eps dz -K'.
\end{equation}
 Then, applying the same method as in the proof of~\autoref{lemmaSemiConvex} (see also the proof of~\autoref{TheoremUeps}), we show
\begin{equation}
W_\eps(t,y) \geq \inf_{y\in\R^n} W_\eps(t=0,y)-K't,\quad \forall\eps>0,\ t\geq 0,\ y\in\R^n.
\end{equation}
Finally,  we conclude the lower bound since $W_\eps(t=0,y)\geq -C$ (from assumption~\eqref{AssumptionInitialSemiConvexity}), we have
\begin{equation}
W_\eps(t,y) \geq -C-K't,\quad \forall\eps>0,\ t\geq 0,\ y\in\R^n.
\end{equation}

\paragraph{Other directional derivatives and conclusion.}

Lower bounds on other second order derivatives in directions  $\nu\in\mathbb{S}:=\{(t,y)\in\R^{n+1}: t^2+\vert y\vert^2=1\}$ can be obtained by a slight adaptation of the previous steps show. We deduce that $U_\eps$ is semi-convex and that $\nabla U_\eps$ is uniformly in $BV_{loc}$ (see Proposition 1.1.3 and Theorem 2.3.1 in~\cite{Cannarsa2004}). We obtain that $U_\eps$ is uniformly bounded in $W_{loc}^{2,1}$.

\subsection{Proof of \autoref{PropoW1estimate}}\label{sec:ProofCorollary}

We recall the definition of $\eta_\eps$ in~\eqref{DefinitionEtaLambda}, and note that differentiating~\eqref{DefU_eps}, we obtain
\begin{equation}\label{LinkVEta}
\D^2_t U_\eps= -\D_t\eta_\eps(t,y)\D_\eta\Lambda(t,\eta_\eps(t,y)), \qquad \D^2_{t,x_i} U_\eps= -\D_{x_i}\eta_\eps(t,y)\D_\eta\Lambda(t,\eta_\eps(t,y)). 
\end{equation}
Using that $U_\eps$ is uniformly bounded in $W^{2,1}$ (\autoref{ThmSemiConvexity}) and that $-\D_\eta\Lambda$ is positively bounded~\eqref{FormuleDeriveeLambda}, we deduce that $\eta_\eps$ is bounded in ${W^{1,1}}$, which proves the first part of \autoref{PropoW1estimate}.

Let us now fix $T>0$ and prove the second part. We recall that $U_\eps$ is semi-concave (\autoref{ThmSemiConvexity}). Thus, there exists a constant $K>0$ such that $\D_t\eta_\eps(t,y)\geq -K$ for all $\eps>0$, $(t,y)\in[0,T]\times\R^n$. Denoting $\overline{\eta}(t):=\sup_{y\in\R^n}\eta(t,y)$, we deduce
$\D_t\overline\eta_\eps(t)\geq -K$ (indeed, the mapping $\theta_y: t\mapsto \eta_\eps(t,y)+Kt$ is nondecreasing, and so is $\sup_{y\in\R^n}\theta_y$). The last inequality should be understood in the sense of the distributions.

We deduce
\begin{equation*}
\int_0^T \vert\D_t\overline\eta_\eps(t)\vert dt
=\int_0^T \D_t\overline\eta_\eps(t) dt+2\int_0^T (\D_t\overline\eta_\eps(t))^- dt
\leq \overline\eta_\eps(t)+2Kt
\leq\overline\eta+2Kt,
\end{equation*}
where $^-$ denotes the negative part, and the last inequality comes from~\eqref{BoundOnEta}. The proof is complete.

\subsection[Asymptotics of $U_\eps$]{Asymptotics of $U_\eps$ - proof of \autoref{TheoremU}}\label{sec:Asymptotics_of_U}
\paragraph{Extraction of a subsequence}
From the a priori estimate of \autoref{TheoremUeps} and Ascoli's theorem, we know that $U_\eps$ converges locally uniformly to some $U$, up to extraction of a subsequence.
Incidentally, this convergence also occurs in $W^{1,1}$, from the $W^{2,1}$ estimate in~\autoref{UepsLipschitz} and a classical compact embedding.
In addition, we know from~\autoref{TheoremUeps} and \autoref{UepsLipschitz} that $U_\eps$ is locally Lipschitz continuous, uniformly in $\eps>0$: for all $t\geq t'\geq0$ and $y,y'\in\R^n$
\begin{equation}\label{LipschitzTbis}
\vert {U}(t,y)-U(t',y')\vert\leq-\underline{\Lambda}(t-t')+ \left(k^0+\frac{L}{l\underline\eta^2} t\right)\vert y-y'\vert .
\end{equation}
Thus, the convergence occurs in $W_{loc}^{1,r}$, $1 \leq r < \infty$.
 Also notice that \autoref{ThmSemiConvexity} implies that $U$ is semi-convex, uniformly in $y\in\R^n$, locally uniformly in $t$.
 
\paragraph{Viscosity solution}
We are going to show that $U$ is a viscosity solution of \eqref{HJU}, i.e., $U$ satisfies
\begin{equation}\label{HJEBIS}
\D_t U=H(y,\nabla_y U),\quad U(0,y)=U^0(y),
\end{equation}
with
\begin{equation}
 H(y,p):=-\Lambda\left(y,\int_{\R^n}M(z)e^{p\cdot z}dz\right).
 \end{equation}
The proof is adapted from classical stability results for viscosity solutions of Hamilton-Jacobi equations (see \cite{GB:94}). However, this case is not completely standard because of the nonlocal term $\int_{\R^n}M(z)e^{\frac{U_\eps(t,y+\eps z)-U_\eps(t,y)}{\eps}}dz$.

\begin{lemma}\label{subsuperBIS}
    The function $U$ is a viscosity solution of~\eqref{HJEBIS} in $(0,\infty)\times\R^n$. Also, for all $T>0$, the viscosity inequalities stand for ${t\in (0,T]}$.
\end{lemma}

\begin{proof}
We are going to prove that $U$ is a subsolution of~\eqref{HJEBIS}. Let us consider a test function $\varphi$ and a point $(t_0,y_0)$ such that $
    {U}-\varphi$ reaches a global maximum at $(t_0,y_0)$. From classical results, there exists $(t_\eps,y_\eps)$ such that 
    \begin{equation*}
        \left\{\begin{aligned}
            &(t_\eps,y_\eps)\underset{\eps \to 0}{\longrightarrow}(t_0,y_0),\\
            &\max_{t,y} U_\eps-\varphi =(U_\eps-\varphi)(t_\eps,y_\eps).
        \end{aligned}\right.
    \end{equation*}
For all $z\in\R^n$, ${\varphi(t_\eps,y_\eps+\eps z)-U_\eps(t_\eps,y_\eps+\eps z)\geq \varphi(t_\eps,y_\eps)-U_\eps(t_\eps,y_\eps)}$, thus we have
                \begin{equation*}
                    \frac{\varphi(t_\eps,y_\eps+\eps z)-\varphi(t_\eps,y_\eps)}{\eps}\geq\frac{U_\eps(t_\eps,y_\eps+\eps z)- U_\eps(t_\eps,y_\eps)}{\eps}.
                \end{equation*}
Since $\D_\eta \Lambda<0$, equation \eqref{HJEBIS} gives
\begin{equation}
\D_t\varphi(t_\eps,y_\eps)=-\Lambda\left(y_\eps,\int_{\R^n} 
M(z)e^{\frac{U_\eps(t_\eps,y_\eps+\eps z)-U_\eps(t_\eps,y_\eps)}{\eps}}d 
z\right)\\
\leq -\Lambda\left(y_\eps,\int_{\R^n} 
M(z)e^{\frac{\varphi(t_\eps,y_\eps+\eps z)-\varphi(t_\eps,y_\eps)}{\eps}}d 
z\right).
\end{equation}
As $\eps$ goes to $0$,
    \begin{equation*}
        \D_t \varphi(t_0,y_0)\leq-\Lambda\left(y_0,\int_{\R^n} M(z)e^{\nabla_y \varphi(t_0,y_0)\cdot z}\right)=H(y_0,\nabla_y \varphi(t_0,y_0)),
    \end{equation*}
then $U$ is a viscosity subsolution of \eqref{HJEBIS}. 
With the same method, we prove that $ U$ is also a viscosity supersolution. It completes the first part of the proof. The second part of the statement is a well-known result, and proof can be found in \cite{GB:94}.
\end{proof}

\paragraph{Uniqueness}

We point out that the Hamiltonian $H$ is Lipschitz continuous in the $y$ variable. We introduce a truncated Hamiltonian
\begin{equation}
\tilde H (y,p):= 
\left\{\begin{aligned}
&H(y,p) &&\text{if }\int_{\R^n}M(z)e^{p\cdot z}dz\in[\underline{\eta},\overline{\eta}],\\
&0 &&\text{otherwise}.
\end{aligned}\right.
\end{equation}
Since $\underline{\eta}\leq\eta_\eps(t,y)\leq\overline{\eta}$ (from~\eqref{BoundOnEta}), we have
\begin{equation}
\D_t U=\tilde H(y,\nabla_y U).
\end{equation}
For this equation, a classical uniqueness result is in order (see e.g~\cite{GB:94,Nordmann2018a}). We deduce that $U_\eps$ converges to $U$ for the whole sequence $\eps\to0$ (and not for an extracted subsequence).

\subsection[A posteriori Lipschitz estimate]{A posteriori Lipschitz estimate - proof of  the global Lipschitz regularity in \autoref{TheoremU}}\label{sec:APosterioriLipschitz}
From \autoref{TheoremUeps} and \autoref{UepsLipschitz}, we know that $U$ is Lipschitz, globally in $t$ and locally in $y$. Our goal is to show that $U$ is globally Lipschitz continuous, i.e., that there exists a constant $C>0$ such that
\begin{equation}\label{GoalAPosterioriLipschitz}
\forall t\geq0,\ \forall (y,y')\in(\R^n)^2, \quad U(t,y)-U(t,y') \leq C\vert y-y'\vert.
\end{equation}
Let us fix $t\geq0$ and $(y,y')\in(\R^n)^2$.

With $\eta_\eps$ defined in~\eqref{DefinitionEtaLambda} and the bound $\eta_\eps\leq\overline{\eta}$ from~\eqref{BoundOnEta}, we have, for all $\eps>0$, $z\in\R^n$
\begin{equation}
\int_{\R^n}M(z)e^{\frac{U_\eps(t,y+\eps z)-U(t,y)}{\eps}}dz\leq \eta_\eps(t,y)\leq \overline{\eta}.
\end{equation}
From the assumption that $M(\cdot)$ is not degenerate, we deduce that, for some $r_0>0$, and for all $z\in\R^n$ such that $\vert z\vert =r_0$, then 
\begin{equation}
\frac{U_\eps(t,y+\eps z)-U_\eps(t,y)}{\eps}\leq C
\end{equation}
for some constant $C$ (independant of $\eps$, $t$, $y$ and $z$). Then, chosing $z$ and $\eps$ such that $y-y'=\eps z,$ we have $U_\eps(t,y)-U_\eps(t,y') \leq C\vert y-y'\vert.$ As $\eps\to0$, we prove the goal~\eqref{GoalAPosterioriLipschitz}.

\bibliographystyle{abbrv}
\bibliography{library.bib}
\end{document}